\documentclass[11pt]{article}

\usepackage{amssymb,amsmath, amscd}
\usepackage{graphicx}

\usepackage{amsfonts,amssymb}
\usepackage[all, knot]{xy}

\usepackage{amsfonts}
\usepackage{amsmath}
\usepackage{amssymb}

\usepackage{chngpage}

\newcommand{\qed}{\hbox{\rule{6pt}{6pt}}}

\newcommand{\Z}{\mathbb{Z}}
\newcommand{\K}{\mathbb{K}}

\def\bar{\overline}
\def\o{\otimes}

\newtheorem{theorem}{Theorem}[section]

\newtheorem{lemma}[theorem]{Lemma}
\newtheorem{proposition}[theorem]{Proposition}
\newtheorem{remark}[theorem]{Remark}

\newtheorem{definition}[theorem]{Definition}

\newtheorem{example}[theorem]{Example}

\newcommand{\Hom}{\mbox{\rm Hom}}

\newenvironment{proof}[1][Proof]{\smallskip\noindent{\bf #1.}\quad}%
{\qed\par\medskip}

\date{}

\begin{document}

\title{ Ternary Distributive Structures and Quandles }

\author{Mohamed Elhamdadi
 \footnote{ Email: \texttt{emohamed@math.usf.edu}}\\
 University of South Florida, USA
 \and
 Matthew Green
 \footnote{ Email: \texttt{ mjgreen@mail.usf.edu}}\\
 University of South Florida, USA
 \and
 Abdenacer Makhlouf
 \footnote{Email: \texttt{Abdenacer.Makhlouf@uha.fr}}\\
 Universit\'e de Haute Alsace, France}

\maketitle

\vspace{10mm}

\begin{abstract}

We introduce  a notion of ternary distributive algebraic structure, give examples, and relate it to the notion of a quandle. Classification is given for low order structures of this type.   Constructions of such structures from ternary bialgebras are provided.  We also describe ternary distributive algebraic structures coming from groups and give examples from vector spaces whose bases are elements of a finite ternary distributive set. We introduce a cohomology theory that is analogous to Hochschild cohomology and relate it to a formal deformation theory  of these structures.

\end{abstract}



\section*{Introduction}

Ternary operations, which are natural generalizations of binary operations,  appear in many areas of mathematics and physics.  An example of a ternary operation of an associative type is a map $\mu$ on a set $X$ satisfying $\mu(\mu(x,y,z),u,v)=\mu(x,\mu(y,z,u),v)=\mu(x,y,\mu(z,u,v))$.  Algebras with these multiplications are called totally associative ternary algebras  and have been considered, for example, in \cite{AM1,AM2}.
The first ternary algebraic structure given in an axiomatic form appeared in 1949 in the work of  N. Jacobson \cite{Jack}.  He considered a Lie bracket $[x,y]$ in a Lie algebra $\cal{L}$ and a subspace that is closed with respect to $[[x,y],z]$ which he called a Lie triple system.  Since then, many works were devoted to ternary structures and their cohomologies (see for example \cite{harris, carlsson, lister, loos, seibt, yamaguti}).  A typical example of  an associative triple system is the ternary algebra of rectangular matrices introduced by M. R. Hestenes \cite{hestenes} where the ternary product is $AB^*C$ (the $*$ stands for the conjugate transpose).  
In theoretical physics, the progress of quantum mechanics, Nambu mechanics, and the work of S. Okubo \cite{Okubo1, Okubo2} allowed an important development in the theory of ternary algebras (see \cite{  AMS1, AMM,  AMS2} for example).  
Furthermore, this generalization of Hamiltonian systems by Nambu generated some profound studies of Nambu-Lie ternary algebras, which are generalizations of Lie algebras.  The algebraic formulation of this structure was achieved by Fillipov \cite{Filippov:nLie} and Takhtajan \cite{takhtajan} based on some generalization of the Jacobi identity.

Distributivity in algebraic structures appeared in many contexts, such as in the quasigroup theory,
the semigroup theory, and the algebraic knot theory.
The notion of a quandle (involutive quandle)  appeared first as an abstraction of
the notion of symmetric transformation, while the racks  were studied
in the context of the conjugation operation in a group. Around 1982, Joyce and Matveev  introduced independently the
notion of a quandle, they associated a quandle  to each oriented knot; this quandle is  called
the knot quandle.  Since then quandles
and racks have been investigated by topologists in order to construct knot and link invariants and
their higher analogues. We  aim in this paper to extend these notions to ternary operations.


 This paper is organized as follows. In Section 1, we will review the basics on quandles and introduce ternary distributive structures  and more generally  $n$-ary distributive operations. We will define naturally the notion of ternary quandles and racks. We will provide some examples and describe some properties.  In section 2, we will give the classification of ternary quandles of low order up to isomorphisms and  we describe ternary distributive algebraic structures coming from groups.  In Section 3,  we will provide some constructions from ternary bialgebras and 3-Lie algebras.  In Section 4 we will
 describe a low dimensional cohomology theory of distributive ternary bialgebra that fits with a deformation theory of ternary distributive operations.
 Section 5  is dedicated to a deformation theory of a weak ternary distributive bialgebra and in particular ternary quandles.

\section{ Quandles and Ternary Distributive Structures}

We begin this section by reviewing the basics of quandles and give examples in order to introduce their analogues in the ternary setting.

A {\it quandle}, $X$, is a set with a binary operation $(a, b) \mapsto a * b$
such that

(1)  For any $a,b,c \in X$, we have
$\quad  (a*b)*c=(a*c)*(b*c). $

(2) For any $a,b \in X$, there is a unique $x \in X$ such that
$a= x*b$.

(3)  For any $a \in X$,
$a* a =a$.\\
 \noindent
Axiom (2) states that for each $u \in X$, the map $R_u:X \rightarrow X$ with $R_u(x):=x*u$ (right multiplication by $u$) is a bijection.  
\noindent
A {\it rack} is a set with a binary operation that satisfies
(1) and (2).

\noindent
Racks and quandles have been studied in, for example,
\cite{FR,Joyce,Matveev}.  The axioms of a quandle (1), (2) and (3) above correspond respectively to the
Reidemeister moves of type III, II, and I.  For more details, see
\cite{FR}, for example. \\
\noindent
Here are some typical examples of quandles.


\begin{list}{--}{}
\item
Any set $X$ with the operation $x*y=x$ for any $x,y \in X$ is
a quandle called the {\it trivial} quandle.
The trivial quandle of $n$ elements is denoted by $T_n$.

\item
A group $X=G$ with
$n$-fold conjugation
as the quandle operation: $a*b=b^{n} a b^{-n}$.

\item
Let $n$ be a positive integer.
For elements
$i, j \in \Z_n$ (integers modulo $n$),
define
$i\ast j \equiv 2j-i \pmod{n}$.
Then $\ast$ defines a quandle
structure  called the {\it dihedral quandle},
  $R_n$.
This set can be identified with  the
set of reflections of a regular $n$-gon
  with conjugation
as the quandle operation.
\item
Any $\Lambda (={\Z }[t, t^{-1}])$-module $M$
is a quandle with
$a*b=ta+(1-t)b$, $a,b \in M$, called an {\it  Alexander  quandle}.
Furthermore for a positive integer
$n$, a {\it mod-$n$ Alexander  quandle}
${\Z }_n[t, t^{-1}]/(h(t))$
is a quandle
for
a Laurent polynomial $h(t)$.
The mod-$n$ Alexander quandle is finite
if the coefficients of the
highest and lowest degree terms
of $h$
  are units in $\Z_n$.
  \end{list}
Now we introduce the analogous notion of a quandle in the ternary setting.

\begin{definition}\label{ternarydef}  {\rm 
Let $Q$ be a set and $T: Q\times Q \times Q \rightarrow Q$ be a ternary operation on $Q$.  The operation  $T$ is said to be {\it right distributive} if it satisfies the following condition
for all  $ x,y,z,u,v \in Q $
\begin{equation}\label{C1-ternary}T(T(x,y,z),u,v)=T(T(x,u,v),T(y,u,v),T(z,u,v)). \quad  \text{(right distributivity)}\end{equation}
}
 \end{definition}

\begin{remark} {\rm
Note that one can similarly define the notions of  left distributive as well as middle distributive.  Through the rest of this paper, we will use distributive to refer to specifically right distributive.\\
Using the diagonal map $D: Q \rightarrow Q \times Q \times Q=Q^{\times 3} $ such that $D(x)=(x,x,x)$, equation  (\ref{C1-ternary}) can be written, as a map from $Q^{\times 5} $ to $Q$,  in the following form 
\begin{equation}\label{eq2}
T\circ (T \times id \times id)=T \circ ( T \times  T \times T)\circ\rho\circ( id \times id \times id \times D \times D).
\end{equation}
where $id$ stands for the identity map and in the whole paper we denote by 
$\rho:Q^{\times 9}\rightarrow Q^{\times 9}$ the map defined as $\rho=p_{6,8}\circ p_{3,7}\circ p_{2,4}$  where $p_{i,j}$ is the transposition $i^{th}$ and $j^{th}$ elements, i.e.
\begin{equation}\label{RHO}
\rho(x_1,\cdots,x_9)=(x_1,x_4,x_7,x_2,x_5,x_8,x_3,x_6,x_9).
\end{equation}  
Equation (\ref{eq2}) as a new form of equation  (\ref{C1-ternary})  will be used in section \ref{bialg} in the context of ternary bialgebras.
}
 \end{remark}
\begin{definition}\label{ternaryquandledef}  {\rm 
Let $T: Q\times Q \times Q \rightarrow Q$ be a  ternary operation on a set $Q$. 
 The pair $(Q,T)$ is said to be \emph{ternary shelf} if $T$ satisfies identity (\ref{C1-ternary}).  If, in addition, for all $a,b \in Q$, the map $R_{a,b}:Q \rightarrow Q$ given by $R_{a,b}(x)=T(x,a,b)$ is invertible,  
  then $(Q,T)$ is said to be \emph{ternary rack}.  If further $T$ satisfies
$  T(x,x,x)=x,$ for all $ x \in Q,$
  then $(Q,T)$ is called a \emph{ternary quandle}.

}
\end{definition}

The figure below is a diagrammatic representation of equation (1).

\vspace{-1.5cm}
\small{
\[\includegraphics[width=0.7\textwidth]{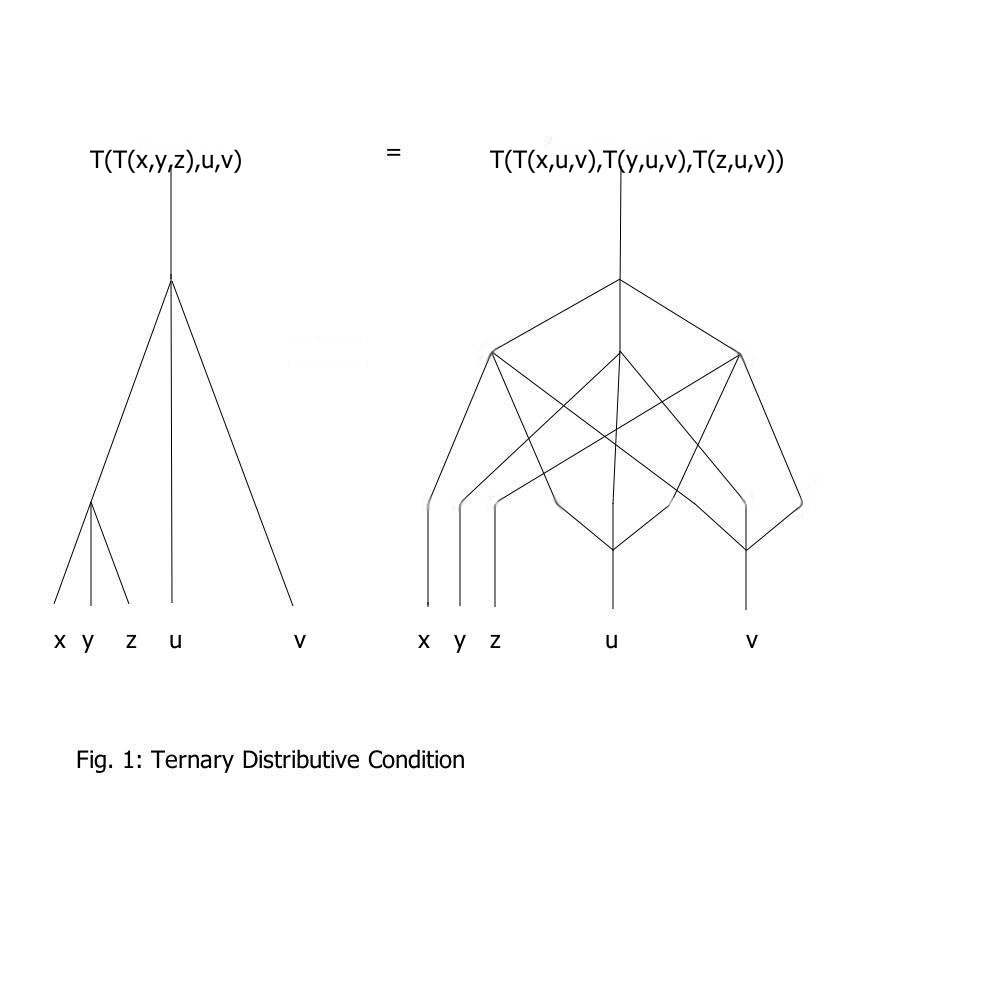}\]
}
\vspace{-3cm}
 \begin{example}
 Let $(Q,*)$ be a quandle and define a ternary operation on $Q$ by $T(x,y,z)=(x * y)* z, \forall x,y,z \in Q$.  It is straightforward to see that $(Q, T)$ is a ternary quandle.  Note that in this case $R_{a,b}=R_b \circ R_a$. We will say that this ternary quandle is induced by a (binary) quandle.

  \end{example}

   \begin{example}
   Let $(M,*)$ be an Alexander quandle, then the ternary quandle coming from $M$ has the operation $T(x,y,z)=T^2x+T(1-T)y+(1-T)z$.
    \end{example}

 \begin{example}
  Let $M$ be any ${\Lambda}$-module where
${\Lambda}=\mathbb{Z}[t^{\pm 1},s]$. The operation $T(x,y,z)=tx+sy+(1-t-s)z$ defines a ternary quandle structure on $M$.  We call this an {\it affine} ternary quandle.
    \end{example}

     \begin{example}
     Consider $\Z_8$ with the ternary operation $T(x,y,z)=3x+2y+4z$.  This affine ternary quandle is not induced by an Alexander quandle structure since $3$ is not a square in $\Z_8$.
     \end{example}

      \begin{example}
  Any group $G$ with the ternary operation $T(x,y,z)=xy^{-1}z$  gives an example of ternary quandle.  This is called $heap$ (sometimes also called a groud) of the group $G$.
     \end{example}

\noindent
A morphism of ternary quandles is a map $\phi: ( Q, T) \rightarrow (Q',T')$ such that $$\phi(T(x,y,z))=T'(\phi(x),\phi(y),\phi(z)).$$  A bijective ternary quandle endomorphism is called ternary quandle automorphism.\\
Therefore, we have a category 
whose objects are ternary quandles and morphisms as defined above.

\begin{definition}  {\rm 
A ternary rack (resp. ternary quandle) $(Q,T)$  is said to be pointed if there is a distinguished element denoted $1 \in Q$ such  that,
for all $x, y  \in Q, T(x,1,1)=x,$ and $ T(1,x,y)=1.$

}
 \end{definition}
 As in the case of the binary quandle there is a notion of {\it medial} ternary quandle 
 
 \begin{definition}  {\rm \cite{Borowic}
A ternary quandle $(Q,T)$  is said to be {\it medial} if for all $a, b, c, d, e, f, g, h, k  \in Q,$ the following identity is satisfied
$$T(T(a,b,c),T(d,e,f),T(g,h,k))= T(T(a,d,g),T(b,e,h),T(c,f,k)).$$
}
 \end{definition}
 This definition of {\it mediality} can be written in term of the following commutative diagram

\[\begin{xy}
    (-2, 20)*+{\underbrace{Q \times \cdot \cdot \cdot \times Q}_{\text{ 9 times}}}="1";
    (-30,10)*+{Q \times Q \times Q}="2";
    (-30,-10)*+{Q}="3";
    (30,-10)*+{Q \times Q \times Q}="5";
    (30, 10)*+{ \underbrace{Q \times \cdot \cdot \cdot \times Q}_{\text{ 9 times}} }="6";
        {\ar^{\rho} "1";"6"};
        {\ar_{T \times T \times T} "1";"2"};
        {\ar_{T} "2";"3"};
        {\ar^{T \times T \times T} "6";"5"};
         {\ar_{T} "5";"3"};

\end{xy}
\]

 Where $\rho=(24)(37)(68)$ is the permutation of the set $\{1,\cdot \cdot \cdot ,9\}$ defined above.
  \begin{example}
  Every affine ternary quandle is medial.       \end{example}

We generalize the notion of ternary quandle to $n$-ary setting.

 \begin{definition}\label{n-ternarydef}  {\rm 
An $n$-ary distributive set is a pair $(Q,T)$ where $Q$ is a set and $T: Q^{ \times n} \rightarrow Q$ is an $n$-ary operation satisfying the following conditions:
\begin{enumerate}
\item
\begin{eqnarray*}\label{C1-n-ary}&T(T(x_1,\cdots,x_n),u_1,\cdots, u_{n-1})=\nonumber \\ & T(T(x_1,u_1,\cdots, u_{n-1}),T(x_2,u_1,\cdots, u_{n-1}),\cdots, T(x_n,u_1,\cdots, u_{n-1})),\end{eqnarray*} $ \forall x_i,u_i \in Q$ (distributivity).

\item
For all $a_1,\cdots,a_{n-1} \in Q$, the map $R_{a_1,\cdots,a_{n-1}}:Q \rightarrow Q$ given by $$R_{a_1,\cdots,a_{n-1}}(x)=T(x,a_1,\cdots,a_{n-1})$$ is invertible.  

\item
For all $x \in Q$,\begin{equation*}\label{C3-n-ary} T(x,\cdots,x)=x.\end{equation*}

\end{enumerate}

If $T$ satisfies only condition(1), then $(Q,T)$ is said to be an $n$-ary shelf.  If both conditions (1) and  (2) are satisfied then $(Q,T)$ is said to be an $n$-ary rack. If all three conditions (1), (2) and (3) are satisfied then $(Q,T)$ is said to be an $n$-ary quandle.

}
 \end{definition}
 
 \begin{definition}  {\rm 
An $n$-ary quandle $(Q,T)$  is said to be {\it medial} if for all $x_{ij}  \in Q, 1 \leq i,j \leq n,$ the following identity is satisfied
\begin{eqnarray*}
T(T(x_{11},x_{12},\cdots, x_{1n}),T(x_{21},x_{22},\cdots, x_{2n}),\cdots, T(x_{n1},x_{n2},\cdots, x_{nn}))=\nonumber \\  
T(T(x_{11},x_{21},\cdots, x_{n1}),T(x_{12},x_{22},\cdots, x_{n2}),\cdots, T(x_{1n},x_{2n},\cdots, x_{nn})).
\end{eqnarray*} 
}
 \end{definition}

\section{Classification of Ternary Quandles of Low Orders}
We give in this section the classification of ternary quandles 
up to isomorphisms. We provide all ternary  quandles of order 2 and 3. Moreover we describe  ternary distributive structures coming from groups. Recall that a ternary quandle is a pair $(Q,T)$, where $Q$ is a set and $T$ a ternary operation,  satisfying the following  conditions 
\begin{enumerate}
\item $T(T(x,y,z),u,v)=T(T(x,u,v),T(y,u,v),T(z,u,v)),$ for all $x,y,z,u,v \in Q$,
 \item  for all $a,b \in Q$, the map $R_{a,b}:Q \rightarrow Q$ given by $R_{a,b}(x)=T(x,a,b)$ is invertible,  
 \item  for all $ x \in Q,$
$  T(x,x,x)=x$.
\end{enumerate}
\subsection{Ternary quandles of order two}
We have the following lemma which states that there are two non-isomorphic ternary quandle structures on a set of two elements.
\begin{lemma}
Let $Q=\{1,2\}$ be a set of two elements.  There are two non-isomorphic ternary quandle structures on $Q$ given by

\begin{eqnarray*}
(i) \quad  T(1,1,1)=1,\quad  T(2,1,1)=2, \quad T(1,1,2)=1, \quad T(2,1,2)=2,\\
T(2,2,2)=2, \quad T( 1,2,2)=1, \quad T(1,2,1)=1, \quad T(2,2,1)=2,
\end{eqnarray*}
 and
\begin{eqnarray*}
 (ii) \quad  T(1,1,1)=1,\quad  T(2,1,1)=2, \quad T(1,1,2)=2, \quad T(2,1,2)=1,\\
T(2,2,2)=2, \quad T( 1,2,2)=1, \quad T(1,2,1)=2, \quad T(2,2,1)=1.
\end{eqnarray*}
\end{lemma}

\begin{proof} Let $Q=\{1,2\}$ and $T$ a ternary quandle operation on $Q$.  Then we have $T(1,1,1)=1$ and $T(2,1,1)=2$.  Similarly, we have $T(2,2,2)=2$ and $T(1,2,2)=1$. Now we need to choose a value for $T(1,1,2)$.
We distinguish two cases:

\noindent
{\bf Case 1}: Assume $T(1,1,2)=1$, this implies $T(2,1,2) =2$ (by second axiom).  We claim that in this case $T(1,2,1)$ can not equal $2$, otherwise $T(2,2,1)=1$ (again axiom (2)).  Now use axiom (3) of right-self-distributivity to get $T(T(2,1,2),2,1)=T( T(2,2,1), T(1,2,1), T(2,2,1) )$ implying that $T(1,2,1)=T(1,2,1)$ but this contradicts the bijectivity of axiom (2).  Then $ T(1,2,1)=1$ and $T(2,2,1)=2$.  This end the proof for case $1$.

\noindent
{\bf Case 2}:  Assume $T(1,1,2)=2$, this implies $T(2,1,2) =1$ (by second axiom).
As in case 1, we prove similarly that $T(1,2,1)$ can not equal $1$, thus $T(1,2,1)=2$ and $T(2,2,1)=1$.  Now, the only non-trivial bijection of the set $\{1,2\}$ is the transposition
sending $1$ to $2$.  It's easy to see that this transposition is not a
homomorphism between the two ternary quandles given in case $1$ and case $2$.
\end{proof}

\subsection{Ternary quandles of order three}

To help classify the ternary quandles two observations proved are useful. First we note that  every ternary quandle is related to some (binary) quandle.

\begin{remark}
If $(Q,T)$ is a ternary quandle, then $(Q,*)$, where $x*y=T(x,y,y)$ is a (binary) quandle.
\end{remark}

We shall refer to this related quandle as the \emph{associated quandle}. We now consider how the relation between associated quandles extends to the ternary quandles.

\begin{lemma}
Let $(Q,T)$ be a ternary quandle, and $(Q,*)$, be the associated quandle defined by $x*y=T(x,y,y)$. If $(R,*')$ is a quandle such that $(Q,*)\cong(R,*')$, then there exists a ternary quandle $(R,T') \cong (Q,T)$ such that $x*'y=T'(x,y,y)$.
\end{lemma}

\begin{proof}
This is easily shown by setting $T'(x,y,z)=\phi(T(\phi^{-1}(x),\phi^{-1}(y),\phi^{-1}(z)))$ where $\phi:Q \rightarrow R$ is an isomorphism from $(Q,*)$ to $(R,*')$.
\end{proof}

With these facts we now see that we may limit the task of generating isomorphically distinct ternary quandles by generating them based on isomorphically distinct quandles. Additionally, it is clear that two different ternary quandles with the same associated quandle will be isomorphically distinct, simplifying our task further.

With these facts as a starting point we developed a simple program using the conditions defining a ternary quandle to compute all ternary quandles of order 3.

Since for each fixed $a,b$, the map $x \mapsto T(x,a,b)$ is a permutation, then in the following table we describe all ternary quandles of order three in terms of the columns of the Cayley table. Each column is a permutation of the elements and is described in standard notation that is by explicitly writing it in terms of products of disjoint cycles.  Thus for a given $z$ we give the permutations resulting from fixing $y=1,2,3$. For example, the ternary set $T_{21}(x,y,z)$ with the Cayley table

\begin{center}
\begin{tabular}{|c|c|c||c|c|c||c|c|c|}
\hline
\multicolumn{3}{|c||}{z=1} &
\multicolumn{3}{|c||}{z=2} &
\multicolumn{3}{|c|}{z=3} \\ \hline
1 & 2 & 3 & 2 & 1 & 1 & 3 & 1 & 1 \\ \hline
2 & 1 & 2 & 1 & 2 & 3 & 2 & 3 & 2 \\ \hline
3 & 3 & 1 & 3 & 3 & 2 & 1 & 2 & 3 \\ \hline
\end{tabular} \end{center}

will be represented with the permutations $(1),(12),(13); (12),(1),(23); (13),(23),(1)$. This will appear on the table as follows.

\footnotesize \begin{center}
\begin{tabular}{|c|c|c|c|}
\hline
$T$ & z=1 & z=2 & z=3 \\ \hline
$T_{21}$ & (1),(12),(13)& (12),(1),(23)& (13),(23),(1)\\ \hline
\end{tabular}
\end{center} \normalsize

Additionally we organize the table based on the associated quandle, given in similar permutation notation.

\footnotesize \begin{center}
\begin{tabular}{|c| *{3}{@{}c@{}|} c|*{3}{@{}c@{}|}}
\hline
\multicolumn{8}{|c|}{Ternary Distributive Sets With Associated Quandle (1),(1),(1)} \\ \hline
$T$ & z=1 & z=2 & z=3 & $T$ & z=1 & z=2 & z=3\\ \hline
$T_{0}$ & (1),(1),(1)& (1),(1),(1)& (1),(1),(1) &
$T_{1}$ & (1),(1),(1)& (1),(1),(1)& (12),(12),(1)\\ \hline
$T_{2}$ & (1),(1),(1)& (1),(1),(23)& (1),(23),(1) &
$T_{3}$ & (1),(1),(1)& (23),(1),(1)& (23),(1),(1)\\ \hline
$T_{4}$ & (1),(1),(1)& (23),(1),(23)& (23),(23),(1)&
$T_{5}$ & (1),(1),(1)& (13),(1),(13)& (1),(1),(1)\\ \hline
$T_{6}$ & (1),(1),(12)& (1),(1),(12)& (1),(1),(1) &
$T_{7}$ & (1),(1),(12)& (1),(1),(12)& (12),(12),(1)\\ \hline
$T_{8}$ & (1),(1),(123)& (123),(1),(1)& (1),(123),(1) &
$T_{9}$ & (1),(1),(132)& (132),(1),(1)& (1),(132),(1)\\ \hline
$T_{10}$ & (1),(1),(13)& (1),(1),(1)& (13),(1),(1) &
$T_{11}$ & (1),(1),(13)& (13),(1),(13)& (13),(1),(1)\\ \hline
$T_{12}$ & (1),(23),(23)& (1),(1),(1)& (1),(1),(1) &
$T_{13}$ & (1),(23),(23)& (1),(1),(23)& (1),(23),(1)\\ \hline
$T_{14}$ & (1),(23),(23)& (23),(1),(1)& (23),(1),(1) &
$T_{15}$ & (1),(23),(23)& (23),(1),(23)& (23),(23),(1)\\ \hline
$T_{16}$ & (1),(23),(23)& (13),(1),(13)& (12),(12),(1) &
$T_{17}$ & (1),(12),(1)& (12),(1),(1)& (1),(1),(1)\\ \hline
$T_{18}$ & (1),(12),(1)& (12),(1),(1)& (12),(12),(1) &
$T_{19}$ & (1),(12),(12)& (12),(1),(12)& (1),(1),(1)\\ \hline
$T_{20}$ & (1),(12),(12)& (12),(1),(12)& (12),(12),(1) &
$T_{21}$ & (1),(12),(13)& (12),(1),(23)& (13),(23),(1)\\ \hline
$T_{22}$ & (1),(123),(1)& (1),(1),(123)& (123),(1),(1) &
$T_{23}$ & (1),(123),(123)& (123),(1),(123)& (123),(123),(1)\\ \hline
$T_{24}$ & (1),(123),(132)& (132),(1),(123)& (123),(132),(1) &
$T_{25}$ & (1),(132),(1)& (1),(1),(132)& (132),(1),(1)\\ \hline
$T_{26}$ & (1),(132),(123)& (123),(1),(132)& (132),(123),(1) &
$T_{27}$ & (1),(132),(132)& (132),(1),(132)& (132),(132),(1)\\ \hline
$T_{28}$ & (1),(13),(1)& (1),(1),(1)& (1),(13),(1) &
$T_{29}$ & (1),(13),(1)& (13),(1),(13)& (1),(13),(1)\\ \hline
$T_{30}$ & (1),(13),(12)& (23),(1),(12)& (23),(13),(1) &
$T_{31}$ & (1),(13),(13)& (1),(1),(1)& (13),(13),(1)\\ \hline
$T_{32}$ & (1),(13),(13)& (13),(1),(13)& (13),(13),(1) & & & &\\ \hline
\multicolumn{8}{|c|}{Ternary Distributive Sets With Associated Quandle (1),(1),(12)} \\
\hline
$T$ & z=1 & z=2 & z=3 & $T$ & z=1 & z=2 & z=3\\ \hline
$T_{33}$ & (1),(1),(1)& (1),(1),(1)& (1),(1),(12) &
$T_{34}$ & (1),(1),(1)& (1),(1),(1)& (12),(12),(12)\\ \hline
$T_{35}$ & (1),(1),(12)& (1),(1),(12)& (1),(1),(12) &
$T_{36}$ & (1),(1),(12)& (1),(1),(12)& (12),(12),(12)\\ \hline
$T_{37}$ & (1),(12),(1)& (12),(1),(1)& (1),(1),(12) &
$T_{38}$ & (1),(12),(1)& (12),(1),(1)& (12),(12),(12)\\ \hline
$T_{39}$ & (1),(12),(12)& (12),(1),(12)& (1),(1),(12) &
$T_{40}$ & (1),(12),(12)& (12),(1),(12)& (12),(12),(12)\\ \hline
\multicolumn{8}{|c|}{Ternary Distributive Sets With Associated Quandle (23),(13),(12)} \\
\hline
$T$ & z=1 & z=2 & z=3 & $T$ & z=1 & z=2 & z=3\\ \hline
$T_{41}$ & (23),(1),(1)& (1),(13),(1)& (1),(1),(12) &
$T_{42}$ & (23),(23),(23)& (13),(13),(13)& (12),(12),(12)\\ \hline
$T_{43}$ & (23),(12),(13)& (12),(13),(23)& (13),(23),(12) &
$T_{44}$ & (23),(123),(132)& (132),(13),(123)& (123),(132),(12)\\ \hline
$T_{45}$ & (23),(132),(123)& (123),(13),(132)& (132),(123),(12) &
$T_{46}$ & (23),(13),(12)& (23),(13),(12)& (23),(13),(12)\\ \hline
\end{tabular}
\end{center} \normalsize

\subsection{Ternary distributive structures from groups}
We search for ternary distributive structures coming from groups. We have the following  necessary condition.

\begin{lemma}
Let $x,y,z$ be three fixed elements in a group $G$.  Let $w(x,y,z)=a_1^{e_1}a_2^{e_2}...a_n^{e_n}$ such that $a_i\in \{ x,y,z\}$ and $e_i=\pm1$. If $w$ is defined such that (I) $\sum^n_{i=1}e_i=1$, (II) there exists a unique $i$ such that $a_i=x$, and (III) w(x,y,z) satisfies equation (\ref{C1-ternary}) of Definition \ref{ternarydef}, then w defines a ternary quandle over the group G.
If $w$ defines a ternary quandle then $\sum^n_{i=1}e_i=1$ and $\sum_{i\in I}e_i=\pm1$ where $I=\{i : a_i=x\}$
\end{lemma}

The condition $\sum^n_{i=1}e_i=1$ is a result of axiom (1) and the condition $\sum_{i\in I}e_i=\pm1$ is a result of axiom (2).

A computer aided search using the sufficient conditions given above gives us the following list of group words of lengths 3, 5 and 7. (note: Condition 1 ensures that all valid words will be of odd length.)\\

{\bf Length 3}

\begin{tabular}{l l l l}
$zxz^{-1}\qquad\qquad\;$ & $yxy^{-1}\qquad\qquad\;$ & $yz^{-1}x\qquad\qquad\;$ & $zy^{-1}x$\\
$xz^{-1}y$ & $y^{-1}xy$ & $xy^{-1}z$ & $z^{-1}xz$
\end{tabular}

{\bf Length 5}

\begin{tabular}{l l l l}
$zzxz^{-1}z^{-1} \qquad\;$ & $zyxy^{-1}z^{-1}\qquad\;$ & $zy^{-1}xyz^{-1}\qquad\;$ & $yzxz^{-1}y^{-1}$\\
$yyxy^{-1}y^{-1}$ & $yz^{-1}xzy^{-1}$ & $yz^{-1}yz^{-1}x$ & $zy^{-1}zy^{-1}x$\\
$yz^{-1}xz^{-1}y$ & $zy^{-1}xz^{-1}y$ & $y^{-1}zxz^{-1}y$ & $xz^{-1}yz^{-1}y$\\
$y^{-1}y^{-1}xyy$ & $y^{-1}z^{-1}xzy$ & $yz^{-1}xy^{-1}z$ & $zy^{-1}xy^{-1}z$\\
$z^{-1}yxy^{-1}z$ & $xy^{-1}zy^{-1}z$ & $z^{-1}y^{-1}xyz$ & $z^{-1}z^{-1}xzz$
\end{tabular}

{\bf Length 7}

\begin{tabular}{l l l l}
$zzzxz^{-1}z^{-1}z^{-1}\;$ & $zzyxy^{-1}z^{-1}z^{-1}\;$ & $zzy^{-1}xyz^{-1}z^{-1}\;$ & $zyzxz^{-1}y^{-1}z^{-1}$\\ $zyyxy^{-1}y^{-1}z^{-1}$ & $zyz^{-1}xzy^{-1}z^{-1}$ & $zy^{-1}zxz^{-1}yz^{-1}$ & $zy^{-1}y^{-1}xyyz^{-1}$\\ $zy^{-1}z^{-1}xzyz^{-1}$ & $yzzxz^{-1}z^{-1}y^{-1}$ & $yzyxy^{-1}z^{-1}y^{-1}$ & $yzy^{-1}xyz^{-1}y^{-1}$\\
$yyzxz^{-1}y^{-1}y^{-1}$ & $yyyxy^{-1}y^{-1}y^{-1}$ & $yyz^{-1}xzy^{-1}y^{-1}$ & $yz^{-1}yxy^{-1}zy^{-1}$\\ $yz^{-1}y^{-1}xyzy^{-1}$ & $yz^{-1}z^{-1}xzzy^{-1}$ & $yz^{-1}yz^{-1}yz^{-1}x$ & $zy^{-1}zy^{-1}zy^{-1}x$\\ $y^{-1}zzxz^{-1}z^{-1}y$ & $y^{-1}zyxy^{-1}z^{-1}y$ & $yz^{-1}yz^{-1}xz^{-1}y$ & $zy^{-1}zy^{-1}xz^{-1}y$\\
$yz^{-1}xz^{-1}yz^{-1}y$ & $zy^{-1}xz^{-1}yz^{-1}y$ & $xz^{-1}yz^{-1}yz^{-1}y$ & $yz^{-1}yx^{-1}yz^{-1}y$\\ $y^{-1}zy^{-1}xyz^{-1}y$ & $y^{-1}y^{-1}zxz^{-1}yy$ & $y^{-1}y^{-1}y^{-1}xyyy$ & $y^{-1}y^{-1}z^{-1}xzyy$\\ $y^{-1}z^{-1}yxy^{-1}zy$ & $y^{-1}z^{-1}y^{-1}xyzy$ & $y^{-1}z^{-1}z^{-1}xzzy$ & $z^{-1}yzxz^{-1}y^{-1}z$\\
$z^{-1}yyxy^{-1}y^{-1}z$ & $yz^{-1}yz^{-1}xy^{-1}z$ & $zy^{-1}zy^{-1}xy^{-1}z$ & $yz^{-1}xy^{-1}zy^{-1}z$\\ $zy^{-1}xy^{-1}zy^{-1}z$ & $xy^{-1}zy^{-1}zy^{-1}z$ & $zy^{-1}zx^{-1}zy^{-1}z$ & $z^{-1}yz^{-1}xzy^{-1}z$\\ $z^{-1}y^{-1}zxz^{-1}yz$ & $z^{-1}y^{-1}y^{-1}xyyz$ & $z^{-1}y^{-1}z^{-1}xzyz$ & $z^{-1}z^{-1}yxy^{-1}zz$\\
$z^{-1}z^{-1}y^{-1}xyzz$ & $z^{-1}z^{-1}z^{-1}xzzz.$ &
\end{tabular}
\normalsize

\begin{remark} {\rm 
In \cite{Maciej}, the author investigated some ternary operations coming from coloring the four regions around crossings of classical knot diagrams.  
We mention that the axioms satisfied by his ternary operations involve only four arguments while our ternary distributive operation axiom  involves five arguments (see equation (\ref{C1-ternary})).
Even though, the ternary operation of the heap of a group $T(x,y,z)=xy^{-1}z$ happened to be an example for both his operations and ours, the difference is that any permutation of the three letters $x,y,z$ in this operation is also an example in his context, while in our situation the ternary operation obtained by the transposition of $x$ and $y$ that is $T(x,y,z)=yx^{-1}z$ is not distributive operation.  This shows that his ternary operations and ours are different.  
}
\end{remark}

 \section{Constructions from Ternary bialgebras and 3-Lie algebras}\label{bialg}
 We provide in this section some constructions of ternary shelves involving ternary bialgebra structures and 3-Lie algebras.

 \subsection{ Ternary bialgebras}
We start by recalling definitions of ternary algebras, coalgebras and bialgebras. See  \cite{Borowic,Duplij,Goze,Zeko} for references about ternary bialgebras.

\begin{definition}\label{ternaryalg}  {\rm
A ternary $\K$-algebra is a triple $(A, \mu, \eta)$ where $A$ is a vector space over a field $\K$ with a multiplication $\mu: A \otimes A \otimes A \rightarrow A$  and a unit $\eta: \K \rightarrow A$ that are linear maps such that  the following associativity identity  is satisfied
\begin{equation}\label{ASSO3} \mu \circ (\mu \otimes id \otimes id)  =  \mu \circ (id \otimes \mu \otimes  id) = \mu \circ (id \otimes  id\otimes \mu)
\end{equation}
and the following property of the unit is also satisfied
$$ \mu \circ (\eta \otimes \eta \otimes id)  =  \mu \circ (\eta \otimes id \otimes  \eta) = \mu \circ (id \otimes  \eta \otimes \eta).$$

The triple $(A, \mu, \eta)$ defines a  {\it weak} ternary $\K$-algebra if, instead of identity  \eqref{ASSO3},  the following weak associativity identity holds 
\begin{equation}\label{ASSO3weak} \mu \circ (\mu \otimes id \otimes id)  =   \mu \circ (id \otimes  id\otimes \mu)
\end{equation}
}
\end{definition}
Ternary coalgebras are defined similarly by changing the directions of the arrows in the previous definition.  Precisely,
\noindent
\begin{definition}\label{ternarycoalg}  {\rm
A vector space $A$ is a ternary $\K$-coalgebra if it has a coalgebra comultiplication $\Delta$ that is a linear map $\Delta:A\rightarrow A\otimes A\otimes A $ satisfying the following  coassociativity identity:
\begin{equation}\label{COASSO3}(\Delta\otimes id \otimes id) \circ  \Delta =(id \otimes \Delta\otimes  id) \circ  \Delta=(id \otimes  id\otimes \Delta) \circ  \Delta.
\end{equation}

\noindent
The Ternary coalgebra is said to be counital if there exists a map $\varepsilon: A \rightarrow\K$ such that

$$(\varepsilon \otimes id \otimes id)\circ \Delta=(id \otimes\varepsilon \otimes  id)\circ \Delta=(id \otimes id \otimes\varepsilon )\circ \Delta.$$
}
The triple $(A, \Delta, \varepsilon)$ defines a ternary weak $\K$-coalgebra if, instead of identity  \eqref{COASSO3},  the following weak coassociativity identity holds 
\begin{equation}\label{COASSO3weak} (\Delta\otimes id \otimes id) \circ  \Delta =(id \otimes  id\otimes \Delta) \circ  \Delta.\end{equation}
\end{definition}

\noindent
A linear map $f:A \rightarrow A$ is called compatible with a comultiplication and the counit $\epsilon$  if $$\Delta\circ f=(f\otimes f\otimes f)\circ \Delta  \;\;\; \mbox{and} \; \;\; \varepsilon f=\varepsilon$$

\begin{definition}\label{ternaryBialg}  {\rm
We say that a ternary operation $\mu:A \otimes A \otimes A \rightarrow A    $ is compatible with a comultiplication $\Delta$ if and only if
\begin{equation}
\Delta\circ \mu=( \mu \otimes  \mu \otimes \mu)\circ\rho\circ( \Delta\otimes\Delta\otimes\Delta),
\end{equation}
where
$\rho:A^{\otimes 9}\rightarrow A^{\otimes 9}$ is defined in equation \eqref{RHO}. }
\end{definition}
 The following is a figure of compatibility of ternary operation with comultiplication
\vspace{-.26cm}
\small{
\[\includegraphics[width=0.7\textwidth]{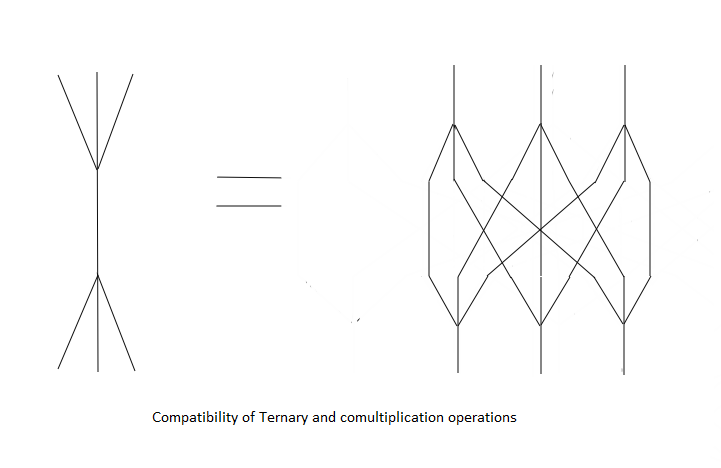}\]

\begin{definition}{\rm
A ternary bialgebra is a quintuplet $(A,\mu,\eta,\Delta, \epsilon)$ such that $(A,\mu,\eta)$ is a ternary algebra, $(A,\Delta, \epsilon)$ is a ternary coalgebra and the multiplication $\mu$ and the unit $\eta$ are coalgebra morphisms (equivalent to $\Delta$ and $\epsilon$ are algebra morphisms).
}
\end{definition}

 \begin{example}
 \begin{itemize}
 \item
     Group Algebras: { \rm Let $G$ be a (multiplicative) group, and $\K [G]$ be the group algebra.  Then $A=\K [G]$ becomes a ternary bialgebra with the multiplication defined  as $T(g,h,k)=gh^{-1}k$, the comultiplication  and counit are given respectively by $\Delta(g)=g \otimes g \otimes g$ and $\epsilon(g)=1$.}
     \item
     Function algebra on groups: { \rm  Let $G$ be a finite group.  Using $\K ^{G \times G \times G}\cong \K ^G \otimes \K^G \otimes \K^G$, the  set $\K ^G$ of functions from $G$ to $\K$ has a ternary  pointwise multiplication and a ternary comultiplication  $\Delta: \K^G \rightarrow \K^{G \times G \times G}$ given by
 $\Delta(f)(u \otimes v  \otimes w)=f(uvw)$.  Now $\K^G$ has basis (the characteristic function) $\delta_g : G
\rightarrow \K$ defined by $\delta_g(x)=1$ if $x=g $ and zero
otherwise. We have
$\displaystyle{\Delta(\delta_h)=\sum_{uvw=h} \delta_u \otimes \delta_v  \otimes \delta_w}$.}

     \end{itemize}
     
     \end{example}

\noindent
Now let $(X,T)$ be a set with a  ternary distributive operation $T$.  Consider
the vector space $V=\K [X]$ with basis the elements of $X$ and let $W=\K\oplus V$. We define $q : W \otimes W \otimes W \rightarrow W$ by linearly extending
\begin{eqnarray*}
& q(x\otimes y\otimes z)=T(x,y,z),  \quad  q(1 \otimes 1 \otimes x)=1,  \\
& q(x \otimes y \otimes 1)=  q(x \otimes 1 \otimes y)=q(x \otimes 1\otimes 1)=q(1 \otimes x \otimes y)=q( 1 \otimes x \otimes 1)=q(1\otimes 1 \otimes 1)=0.
\end{eqnarray*}
for all $x,y, z \in X$.

More explicitly, if we write  elements of $W$ in the form  $a+\sum a_x x$, where $a,a_x\in \K $ and $x\in X$, then
\begin{equation}
q((a+\sum a_x x)\otimes (b+\sum b_y y)\otimes (c+\sum c_z z))=\sum_{x,y,z}a_x b_y c_zT(x,y,z) +  \sum_z abc_z.
\end{equation}

  \begin{proposition}
  The map $q$ defined above satisfies the condition \eqref{C1-ternary} that is
   \begin{equation*}
q \circ (q \otimes id \otimes id)=q \circ (q \otimes q \otimes q)\circ \rho\circ (id \otimes    id \otimes id \otimes D \otimes D).
\end{equation*}
where $D$ is map (extended linearly) that sends $x \in X$ to $x \otimes x \otimes x \in V \otimes V \otimes V$.
  \end{proposition}

  \begin{proof}
  By calculation, the LHS of this equation is
  \begin{multline*}
  q(q\otimes id \otimes id)((a+\sum a_x x)\otimes (b+\sum b_y y)\otimes (c+\sum c_z z)\otimes (d+\sum_u d_u u) \otimes (f+\sum_v f_v v) )\\
  =q(\sum_z abc_z+\sum_{x,y,z}a_x b_y c_zT(x,y,z)\otimes (d+\sum_u d_u u) \otimes (f+\sum_v f_v v) )\\
  =\sum_{z,v} abc_zdf_v + \sum_{x,y,z,u,v}a_x b_y c_z d_u f_v T(T(x,y,z),u,v).
  \end{multline*}
  From the RHS we obtain
  \begin{multline*}
  q  (q \otimes q \otimes q) \rho  (id \otimes    id \otimes id \otimes D \otimes D)
((a+\sum a_x x)\otimes (b+\sum b_y y)\otimes (c+\sum c_z z)\otimes (d+\sum_u d_u u) \otimes (f+\sum_v f_v v) )\\
  =\sum_{z,v} abc_zdf_v + \sum_{x,y,z,u,v}a_x b_y c_z d_u f_v T(T(x,u,v),T(y,u,v),T(z,u,v)).    \end{multline*}
  This ends the proof.
 \end{proof}

  \begin{proposition}
  The map $q$ is compatible with the comultiplication and not compatible with the counit obtained by linearly extending $\varepsilon(1)=1$.
  \end{proposition}

  \begin{proof}
  Again by calculation
  \begin{multline}
D\circ q((a+\sum a_x x)\otimes (b+\sum b_y y)\otimes (c+\sum c_z z))\\
  =q((a+\sum a_x x)\otimes (b+\sum b_y y)\otimes (c+\sum c_z z))^3\\
  =(q\otimes q\otimes q)\circ\rho\circ( D\otimes D\otimes D)\circ((a+\sum a_x x) \otimes (b+\sum b_y y)\otimes (c+\sum c_z z)).
  \end{multline}
  Thus we see $q$ is compatible with comultiplication.

\noindent
  As for composition with the counit we see
  \begin{equation}
  (\varepsilon\circ q) [(a+\sum a_x x)\otimes (b+\sum b_y y)\otimes (c+\sum c_z z)]=\sum_z abc_z+\sum_{x,y,z}a_x b_y c_z.
  \end{equation}
  This implies that $\varepsilon\circ q \neq \varepsilon,$ so that the map is not compatible with the counit.
  \end{proof}

 \subsection{$3$-Lie algebras}
In the following we recall the definition of a $3$-Lie algebra and show how to derive a ternary distributive operation from it. For further properties and results about $3$-Lie algebras, we refer the reader to  \cite{AMS1,AMM,AMS2,aip:review,Filippov:nLie,Okubo1,Okubo2}.
 
 \begin{definition}
  A $3$-Lie algebra is a $\K$-vector space $L$  together with  a skewsymmetric  ternary operation $[\cdot ,\cdot  ,\cdot   ]$  satisfying
  \begin{equation*}\label{TernaryNambuIdentityOrdinary}
  [ [x_{1},x_{ 2}, x_{3}],x_{4},x_{5}] =
  [ [x_1,x_4,x_5 ],x_{2},x_{3}]
  + [ x_1,[ x_2,x_4,x_{5}] , x_{3}]+
  [x_1,x_{2},[x_3,x_4,x_{5}]].
\end{equation*}
\end{definition}
Given a $3$-Lie algebra $L$, we can construct a ternary coalgebra  $N=\K\oplus L$ by setting for all $ x \in L$,
\begin{equation}\label{comultip}
\Delta(x)=x\otimes 1\otimes 1+1\otimes x\otimes 1+1\otimes 1\otimes x, \  \Delta(1)=1\otimes 1\otimes 1, \ \varepsilon(1)=1 \text{ and } \varepsilon(x)=0.
\end{equation}

  We define a linear map $T: N \otimes N \otimes N \rightarrow N$, for all  $x,y,z\in L,$ by
  \begin{eqnarray*}
&&T(x\otimes y\otimes z)=[x,y,z], \quad T(1\otimes 1\otimes 1)=1, \quad T(x\otimes 1\otimes 1)=x, \\
&& T(x\otimes 1\otimes y)=q(1\otimes x\otimes 1)=T(x\otimes y\otimes 1)=T(1 \otimes x\otimes y)=T(1\otimes 1\otimes x)=0.
\end{eqnarray*}

That is  for
$a, b,c \in \K$
$$
q((a+x ) \otimes (b+y) \otimes (c+z) ) = abc + bcx + [x,y,z]
= (abc, bcx + [x,y,z]).$$

\noindent
As in the previous situation we have
\begin{proposition}
The map $T$ defined above satisfies the ternary distributive condition \eqref{C1-ternary}.
  \end{proposition}
  
  \begin{proof}
  In a similar manner to the proof of the previous theorem, we write explicitly LHS and RHS, 
  \begin{multline*}
 LHS= T(T\otimes id \otimes id)((a+\sum a_x x)\otimes (b+\sum b_y y)\otimes (c+\sum c_z z)\otimes (d+\sum_u d_u u) \otimes (f+\sum_v f_v v) )\\
 RHS= T (T \otimes T \otimes T) \rho  (id \otimes    id \otimes id \otimes D \otimes D)
((a+\sum a_x x)\otimes (b+\sum b_y y)\otimes (c+\sum c_z z)\otimes (d+\sum_u d_u u) \otimes (f+\sum_v f_v v) )\\ 
   \end{multline*}
  By expanding and equating them we obtain the result.  
 \end{proof}
 \noindent
 A direct verification shows that this map $T$ also satisfies the equation $\Delta T=(T\otimes T \otimes T) \rho (D  \otimes D \otimes D)$ giving the following
 \begin{proposition}
The map $T$ defined above is a coalgebra homomorphism.
  \end{proposition}

\subsection{Distributive ternary bialgebras}
We introduce here a notion of distributive ternary bialgebra.
\begin{definition}{\rm
A ternary bialgebra (resp. ternary weak bialgebra) is a triple $(A,T,\Delta )$,  where $A$ is a vector space, $T:A^{\otimes 3}\rightarrow A$ a ternary operation, $\Delta:A\rightarrow A^{\otimes 3}$ a ternary comultiplication,  such that 
\begin{enumerate}
\item the operation $T$ is distributive, meaning it satisfies equation  \eqref{C1-ternary},
\item the comultiplication $\Delta$ is  coassociative (resp. weak coassociative),
\item the maps $T$ and $\Delta$ are compatible, that is
\begin{equation}
\Delta\circ T=( T \otimes  T \otimes T)\circ\rho\circ( \Delta\otimes\Delta\otimes\Delta).
\end{equation}
\end{enumerate}
}
\end{definition}

It turns out that any  ternary operation $T$  over   a set $X$  endows $\K [X]$ with a structure of distributive ternary bialgebra with a comultiplication $D$ defined on the generators as $D(x)=x\otimes x\otimes x$ for all $x\in X$, and then extended linearly. Indeed, $D$ is coassociative and compatible with $T$. 
We have for $x,y,z\in X$
\begin{eqnarray*}
 ( T \otimes  T \otimes T)\circ\rho\circ( D \otimes D\otimes D)(x\otimes y\otimes z)& =&( T \otimes  T \otimes T)(\rho(x\otimes x\otimes x\otimes y\otimes y\otimes y\otimes z \otimes z\otimes z)))\\
&=&
( T \otimes  T \otimes T)(x\otimes y\otimes z\otimes x\otimes y\otimes z\otimes x\otimes y\otimes z)\\
&=&
D\circ T(x\otimes y\otimes z).
\end{eqnarray*}

\section{Differentials and Cohomology}
In this section, we provide a cohomology theory of distributive ternary bialgebras that fits with a deformation theory of ternary distributive operations.

Let $(Q,T)$ be a ternary distributive set, $\K$ be an algebraically closed field of
characteristic zero and let $A=\K[Q]$ be the vector space spanned by the elements of $Q$. We extend by bilinearity the ternary distributive operation $T$ to $A$. \\
We define low dimensional cochain groups by
$$\mathcal{C}^1:=Hom(A,A), \quad \mathcal{C}^2:= Hom(A^{\otimes 3} ,A) \oplus Hom(A,A^{\otimes 3}),$$ and $$\mathcal{C}^3:=Hom(A^{\otimes 5} ,A)  \oplus Hom(A^{\otimes 3},A^{\otimes 3}) \oplus Hom(A,A^{\otimes 5}).$$

\noindent
{\bf First differentials:}
  For $f \in \mathcal{C}^1$, we define  \begin{eqnarray*}
\delta_{m} ^1(f):= T (f \otimes id \otimes id)+T (id \otimes f \otimes id) + T(id \otimes id \otimes f) - f T,
\end{eqnarray*}
and
 \begin{eqnarray*}
\delta_{c} ^1(f):=  (f \otimes id \otimes id)  \Delta +  (id \otimes f \otimes id)  \Delta +  ( id \otimes id \otimes f)  \Delta - \Delta f,
\end{eqnarray*}
The differential $\mathcal{D}^1$ is given by $$\mathcal{D}^1(f):= \delta_{m} ^1(f) -  \delta_{c} ^1(f): \mathcal{C}^1=Hom(A,A) \rightarrow Hom(A^{\otimes 3} ,A) \oplus Hom(A,A^{\otimes 3})=\mathcal{C}^2.$$
{\bf Second differentials:} we define the second differentials for $\psi_1 \in\mathcal{C}^1$ and $\psi_2 \in\mathcal{C}^2$ by

\begin{eqnarray*}
 d^{2,1}(\psi_1,\psi_2) &=& [{\psi_1} ({T} \otimes id \otimes id ) + {T} (\psi_1 \otimes id \otimes id ) ]- [\psi_1 ( {T} \otimes  {T} \otimes {T})\circ\rho\circ( id \otimes id \otimes id  \otimes {\Delta}  \otimes {\Delta})  +\\
&&  {T}( \psi_1 \otimes  {T} \otimes {T})\circ\rho\circ( id \otimes id \otimes id  \otimes {\Delta}  \otimes {\Delta}) +  {T}( {T} \otimes  \psi_1 \otimes {T})\circ\rho\circ( id \otimes id \otimes id  \otimes {\Delta}  \otimes {\Delta}) + \\
&&  {T}( {T} \otimes  {T} \otimes \psi_1)\circ\rho\circ( id \otimes id \otimes id  \otimes {\Delta}  \otimes {\Delta}) +  {T}( {T} \otimes  {T} \otimes {T})\circ\rho\circ( id \otimes id \otimes id  \otimes \psi_2   \otimes {\Delta}) +\\
&& {T}( {T} \otimes  {T} \otimes {T})\circ\rho\circ( id \otimes id \otimes id  \otimes {\Delta}  \otimes \psi_2 ) ],\\
 d^{2,2}(\psi_1,\psi_2)  &=& [ \psi_2 T  + \Delta \psi_1  ]- [( \psi_1 \otimes  T \otimes T)\circ\rho\circ( \Delta  \otimes \Delta  \otimes \Delta) + ( T \otimes  \psi_1 \otimes T)\circ\rho\circ( \Delta  \otimes \Delta  \otimes \Delta) +\\
 && ( T \otimes  T \otimes \psi_1)\circ\rho\circ( \Delta  \otimes \Delta  \otimes \Delta) +( T \otimes  T \otimes T)\circ\rho\circ( \psi_2  \otimes \Delta  \otimes \Delta) +\\
 && ( T \otimes  T \otimes T)\circ\rho\circ( \Delta  \otimes \psi_2 \otimes \Delta) + ( T \otimes  T \otimes T)\circ\rho\circ( \Delta  \otimes \Delta  \otimes \psi_2) ],\\
  d^{2,3}(\psi_1,\psi_2)  &=& [ (\psi_2 \otimes  id \otimes id ) \Delta + (\Delta \otimes  id \otimes id ) \psi_2 ] - [( id \otimes id  \otimes \psi_2) \Delta + ( id \otimes id  \otimes \Delta) \psi_2].
\end{eqnarray*}
Remark that this last formula of  $d^{2,3}(\psi_1,\psi_2)$ involves only $\psi_2$ and it corresponds to the co-Hochschild  $2$-differential of the ternary comultiplication.\\
The differential $\mathcal{D}^2$ is given by
$$\mathcal{D}^2(\psi_1,\psi_2):=  d^{2,1}(\psi_1,\psi_2) + d^{2,2}(\psi_1,\psi_2) +  d^{2,3}(\psi_1,\psi_2). $$

\begin{proposition}{\it
The composite $\mathcal{D} ^{2}\circ \mathcal{D}^{1}$ is zero.}
\end{proposition}
{\it Proof.}
By assumption we have  $\mathcal{D}^1(f)= (\psi_1, \psi_2)$, where $\psi_1=\delta^{1}_m(f)$ and $\psi_2=\delta^{1}_c(f)$.   First we prove that $d^{2,1}(\mathcal{D}^1(f))=0$, then exactly using the same technics we obtain that  $d^{2,2}(\mathcal{D}^1(f))=0$.  Since $d^{2,3}$ corresponds to the coHochschild  $2$-differential for the ternary comultiplication $\Delta$, it's straightforward that   $d^{2,3}(\mathcal{D}^1(f))=0$.  Now the details:\\

We write $d^{2,1}(\mathcal{D}^1(f))$ using eight terms $T_1, \cdot \cdot \cdot, T_8$ as follows
\begin{eqnarray*}
d^{2,1}(\mathcal{D}^1(f))&=&[T_1 +T_2] -[T_3+T_4+T_5+T_6+T_7+T_8]
\end{eqnarray*}
where
\begin{eqnarray*}
T_1&=& T (f \otimes id \otimes id) ({T} \otimes id \otimes id )+T (id \otimes f \otimes id)({T} \otimes id \otimes id ) +\\
&& T(id \otimes id \otimes f)({T} \otimes id \otimes id )
 - f T ({T} \otimes id \otimes id ),\\\\
 T_2&=& T(T (f \otimes id \otimes id) \otimes id \otimes id)+T(T (id \otimes f \otimes id)\otimes id \otimes id) +\\
 && T( T(id \otimes id \otimes f) \otimes id \otimes id)- T(f T \otimes id \otimes id),
\end{eqnarray*}
\begin{eqnarray*}
 T_3&=&  T (f \otimes id \otimes id)( {T} \otimes  {T} \otimes {T})\circ\rho\circ( id \otimes id \otimes id  \otimes {\Delta}  \otimes {\Delta})+\\
 &&T (id \otimes f \otimes id)( {T} \otimes  {T} \otimes {T})\circ\rho\circ( id \otimes id \otimes id  \otimes {\Delta}  \otimes {\Delta}) +\\
 &&T(id \otimes id \otimes f)( {T} \otimes  {T} \otimes {T})\circ\rho\circ( id \otimes id \otimes id  \otimes {\Delta}  \otimes {\Delta}) -\\
 &&f T( {T} \otimes  {T} \otimes {T})\circ\rho\circ( id \otimes id \otimes id  \otimes {\Delta}  \otimes {\Delta}),\\\\
 T_4&=&  T(T (f \otimes id \otimes id)\otimes  {T} \otimes {T})\circ\rho\circ( id \otimes id \otimes id  \otimes {\Delta}  \otimes {\Delta})+\\
 &&T(T (id \otimes f \otimes id)\otimes  {T} \otimes {T})\circ\rho\circ( id \otimes id \otimes id  \otimes {\Delta}  \otimes {\Delta}) +\\
 &&T( T(id \otimes id \otimes f)\otimes  {T} \otimes {T})\circ\rho\circ( id \otimes id \otimes id  \otimes {\Delta}  \otimes {\Delta}) -\\
 &&T( f T\otimes  {T} \otimes {T})\circ\rho\circ( id \otimes id \otimes id  \otimes {\Delta}  \otimes {\Delta}),
\end{eqnarray*}
\begin{eqnarray*}
 T_5&=&  {T}( {T} \otimes  T (f \otimes id \otimes id) \otimes {T})\circ\rho\circ( id \otimes id \otimes id  \otimes {\Delta}  \otimes {\Delta})+\\
 &&{T}( {T} \otimes  T (id \otimes f \otimes id) \otimes {T})\circ\rho\circ( id \otimes id \otimes id  \otimes {\Delta}  \otimes {\Delta}) + \\
 &&{T}( {T} \otimes  T(id \otimes id \otimes f) \otimes {T})\circ\rho\circ( id \otimes id \otimes id  \otimes {\Delta}  \otimes {\Delta}) - \\
 &&{T}( {T} \otimes  f T \otimes {T})\circ\rho\circ( id \otimes id \otimes id  \otimes {\Delta}  \otimes {\Delta}),
\\\\
 T_6&=&  {T}( {T} \otimes  {T} \otimes T (f \otimes id \otimes id))\circ\rho\circ( id \otimes id \otimes id  \otimes {\Delta}  \otimes {\Delta})+\\
 &&{T}( {T} \otimes  {T} \otimes T (id \otimes f \otimes id))\circ\rho\circ( id \otimes id \otimes id  \otimes {\Delta}  \otimes {\Delta}) +\\
 &&{T}( {T} \otimes  {T} \otimes T(id \otimes id \otimes f))\circ\rho\circ( id \otimes id \otimes id  \otimes {\Delta}  \otimes {\Delta}) -\\
 &&{T}( {T} \otimes  {T} \otimes f T)\circ\rho\circ( id \otimes id \otimes id  \otimes {\Delta}  \otimes {\Delta}),
\end{eqnarray*}
\begin{eqnarray*}
 T_7&=&{T}( {T} \otimes  {T} \otimes {T})\circ\rho\circ( id \otimes id \otimes id  \otimes (f \otimes id \otimes id)  \Delta  \otimes {\Delta}) + \\
 && {T}( {T} \otimes  {T} \otimes {T})\circ\rho\circ( id \otimes id \otimes id  \otimes (id \otimes f \otimes id)  \Delta  \otimes {\Delta})+  \\
 &&{T}( {T} \otimes  {T} \otimes {T})\circ\rho\circ( id \otimes id \otimes id  \otimes ( id \otimes id \otimes f)  \Delta \otimes {\Delta}) -\\
 &&{T}( {T} \otimes  {T} \otimes {T})\circ\rho\circ( id \otimes id \otimes id  \otimes \Delta f  \otimes {\Delta}),\\\\
 T_8&=&{T}( {T} \otimes  {T} \otimes {T})\circ\rho\circ( id \otimes id \otimes id  \otimes {\Delta}  \otimes(f \otimes id \otimes id)  \Delta )+ \\
 &&{T}( {T} \otimes  {T} \otimes {T})\circ\rho\circ( id \otimes id \otimes id  \otimes {\Delta}  \otimes(id \otimes f \otimes id)  \Delta )+  \\
 &&{T}( {T} \otimes  {T} \otimes {T})\circ\rho\circ( id \otimes id \otimes id  \otimes {\Delta}  \otimes( id \otimes id \otimes f)  \Delta )- \\
 &&{T}( {T} \otimes  {T} \otimes {T})\circ\rho\circ( id \otimes id \otimes id  \otimes {\Delta}  \otimes\Delta f ) 
\end{eqnarray*}

Let $T_{i,j}$ represents the $j$-th term of $T_i$ (in the order given).  Via the ternary distributive equation (\ref{C1-ternary}) and simple cancellation, the terms can be shown to cancel (in pairs) as follows: $T_{1,1}$ \& $T_{2,4}$, $T_{1,2}$ \& $T_{7,4}$, $T_{1,3}$ \& $T_{8,4}$, $T_{1,4}$ \& $T_{3,4}$, $T_{2,1}$ \& $T_{4,1}$, $T_{2,2}$ \& $T_{5,1}$, $T_{2,3}$ \& $T_{6,1}$, $T_{3,1}$ \& $T_{4,4}$, $T_{3,2}$ \& $T_{5,4}$, $T_{3,3}$ \& $T_{6,4}$, $T_{4,2}$ \& $T_{7,1}$, $T_{4,3}$ \& $T_{8,1}$, $T_{5,2}$ \& $T_{7,2}$, $T_{5,3}$ \& $T_{8,2}$, $T_{6,2}$ \& $T_{7,3}$, $T_{6,3}$ \& $T_{8,4}$, so that we then obtain at the end $\mathcal{D} ^{2}(\mathcal{D}^{1}(f))=0.$  As we mentioned before using exactly the same technics we obtain that  $d^{2,2}(\mathcal{D}^1(f))=0$.  Since $d^{2,3}$ corresponds to the coHochschild  $2$-differential for the ternary multiplication $\Delta$, it's straightforward that $d^{2,3}(\mathcal{D}^1(f))=0$.  Thus $\mathcal{D} ^{2}\circ \mathcal{D}^{1}$ is zero.
$\Box$\\
\noindent
The   $1$-cocycle spaces of $A$ are
\begin{equation*}
\mathit{Z}_m^{1}(A ,A)=\{f: A\rightarrow
A\,:\,\delta_m^{1}f=0\}, \;\; \mathit{Z}_c^{1}(A ,A)=\{f:  A\rightarrow
A\,:\, \delta_c^{1}f=0\} 
\end{equation*}
and
\begin{equation*}
 \mathit{Z}^{1}(A ,A)=  \mathit{Z}_m^{1}  \oplus \mathit{Z}_c^{1}=H^1(A,A).
\end{equation*}
The   $2$-coboundaries space of $A$ is
\begin{equation*}
\mathit{B}^{2}(A ,A ) =Im(\mathcal{D}^1).
\end{equation*}
The   $2$-cocycles space of ${A}$ is
\begin{equation*}
\mathit{Z}^{2}({A,A}) =ker(\mathcal{D}^2).
\end{equation*}
Then the second cohomology group is given by the quotient $\mathit{Z}^{2}({A,A})/\mathit{B}^{2}({A,A})$.

 \section{One-parameter formal deformations}
In this section we extend to ternary distributive sets the theory of deformation of rings and associative algebras introduced by Gerstenhaber \cite{Gerst1} and by Nijenhuis and Richardson for Lie
algebras \cite{NR}.  The fundamental results of Gerstenhaber's theory connect deformation theory with the suitable cohomology groups.  This theory was extended to ternary algebras of associative type in \cite{AM1,AM2}.

\noindent
In the following
we define the concept of deformation for a ternary distributive bialgebra and provide the connection to 
cohomology groups. The idea is to deform both the ternary multiplication and the ternary comultiplication at the same time.


\noindent
Let  $(A, T, \Delta) $ be a distributive ternary bialgebra.  A deformation of $(A, T, \Delta) $ is a
$\K [[t]]$-bialgebra $(A_t, T_t, \Delta_t)$, where $A_t=A \otimes \K [[ t ]]$
and $ A_t/(tA_t) \cong A$. Deformations of $T$ and  $\Delta$ are given by
$T_t= T + t T_1 + \cdots + t^n T_n + \cdots : A_t \otimes A_t \otimes A_t \rightarrow A_t$
and $\Delta_t = \Delta + t \Delta_1
+ \cdots + t^n \Delta_n + \cdots : A_t \rightarrow A_t  \otimes A_t \otimes A_t$
where $T_i: A \otimes A  \otimes A \rightarrow A $, $\Delta_i : A \rightarrow A  \otimes A  \otimes A$,
$i=1, 2, \cdots$, are sequences of maps.

\noindent Suppose $\bar{T}=T + t T_1 + \cdots + t^n T_n$ and
$\bar{\Delta} =\Delta + t \Delta_1+ \cdots + t^n \Delta_n$ satisfy the bialgebra
conditions (distributivity,  coassociativity and compatibility) mod $t^{n+1}$,
and suppose
that  there exist $T_{n+1}: A \otimes A  \otimes A \rightarrow A$ and
$\Delta_{n+1}: A \rightarrow A \otimes A  \otimes A$ such that
$\bar{T}+t^{n+1} T_{n+1}$ and $\bar{\Delta}+ t^{n+1} \Delta_{n+1}$
satisfy the bialgebra conditions mod $t^{n+2}$.
Define  $\phi_1 \in \Hom(A^{\otimes 5}, A)$,
$\phi_2\in \Hom(A^{\otimes 3}, A^{\otimes 3})$,
and $\phi_3 \in \Hom(A, A^{\otimes 5})$
by
\begin{eqnarray}
\bar{T} (\bar{T} \otimes id \otimes id ) - \bar{T}( \bar{T} \otimes  \bar{T} \otimes \bar{T})\circ\rho\circ( id \otimes id \otimes id  \otimes \bar{\Delta}  \otimes \bar{\Delta}) &=&  t^{n+1} \phi_1 \quad {\rm mod}\  t^{n+2} , \label{hoch2d1} \\
\bar{\Delta}\;\; \bar{T} - ( \bar{T} \otimes  \bar{T} \otimes \bar{T})\circ\rho\circ( \bar{\Delta}  \otimes \bar{\Delta}  \otimes \bar{\Delta})
&=&  t^{n+1} \phi_2 \quad {\rm mod}\  t^{n+2} , \label{hoch2d2}\\
(\bar{\Delta} \otimes  id \otimes id ) \bar{\Delta} - ( id \otimes id  \otimes \bar{\Delta})\bar{\Delta}
&=&  t^{n+1} \phi_3 \quad {\rm mod}\  t^{n+2} . \label{hoch2d3}
\end{eqnarray}

\noindent
Now expanding these three equations we obtain the values of $\phi_1, \phi_2 $ and $\phi_{3}:$

\begin{eqnarray*}
 \phi_1 &=& [{T_{n+1}} ({T} \otimes id \otimes id ) + {T} ({T_{n+1}} \otimes id \otimes id ) ]- [{T_{n+1}}( {T} \otimes  {T} \otimes {T})\circ\rho\circ( id \otimes id \otimes id  \otimes {\Delta}  \otimes {\Delta})  +\\
&&  {T}( {T_{n+1}} \otimes  {T} \otimes {T})\circ\rho\circ( id \otimes id \otimes id  \otimes {\Delta}  \otimes {\Delta}) +  {T}( {T} \otimes  {T_{n+1}} \otimes {T})\circ\rho\circ( id \otimes id \otimes id  \otimes {\Delta}  \otimes {\Delta}) \\
&&  {T}( {T} \otimes  {T} \otimes {T_{n+1}})\circ\rho\circ( id \otimes id \otimes id  \otimes {\Delta}  \otimes {\Delta}) +  {T}( {T} \otimes  {T} \otimes {T})\circ\rho\circ( id \otimes id \otimes id  \otimes {\Delta_{n+1}}  \otimes {\Delta}) +\\
&& {T}( {T} \otimes  {T} \otimes {T})\circ\rho\circ( id \otimes id \otimes id  \otimes {\Delta}  \otimes {\Delta_{n+1}}) ],\\
 \phi_2 &=& [ \Delta_{n+1} T  + \Delta T_{n+1}  ]- [( T_{n+1} \otimes  T \otimes T)\circ\rho\circ( \Delta  \otimes \Delta  \otimes \Delta) + ( T \otimes  T_{n+1} \otimes T)\circ\rho\circ( \Delta  \otimes \Delta  \otimes \Delta) +\\
 && ( T \otimes  T \otimes T_{n+1})\circ\rho\circ( \Delta  \otimes \Delta  \otimes \Delta) +( T \otimes  T \otimes T)\circ\rho\circ( \Delta_{n+1}  \otimes \Delta  \otimes \Delta) +\\
 && ( T \otimes  T \otimes T)\circ\rho\circ( \Delta  \otimes \Delta_{n+1}  \otimes \Delta) + ( T \otimes  T \otimes T)\circ\rho\circ( \Delta  \otimes \Delta  \otimes \Delta_{n+1}) ],\\
  \phi_3 &=& [ (\Delta_{n+1} \otimes  id \otimes id ) \Delta + (\Delta \otimes  id \otimes id ) \Delta_{n+1} ] - [( id \otimes id  \otimes \Delta_{n+1}) \Delta + ( id \otimes id  \otimes \Delta) \Delta_{n+1}].
\end{eqnarray*}

\begin{proposition}
Let  $(A, T, \Delta) $ be a distributive ternary weak bialgebra and  $(A_t, T_t, \Delta_t)$, where $A_t=A \otimes \K [[ t ]]$,
$T_t= T + t T_1 + \cdots + t^n T_n + \cdots : A_t \otimes A_t \otimes A_t \rightarrow A_t$
and $\Delta_t = \Delta + t \Delta_1
+ \cdots + t^n \Delta_n + \cdots : A_t \rightarrow A_t  \otimes A_t \otimes A_t$
where $T_i: A \otimes A  \otimes A \rightarrow A $, $\Delta_i : A \rightarrow A  \otimes A  \otimes A$,
$i=1, 2, \cdots$, are sequences of maps.

Then  $\mathcal{D}^2(T_{1}, \Delta_{1})=(\phi_1, \phi_2, \phi_{3})=0$.
\end{proposition}

In the sequel, we focus on  deformations of a  ternary distributive set   $(Q,T)$ and set   $A=\K[Q]$ to  be the vector space spanned by the elements of $Q$. We refer to $(\K [Q],T)$, where $T$ is extend by $\K$-trilinearity, as a ternary distributive algebra. 

\noindent
\begin{definition} {\rm
A one-parameter formal deformation of $(\K[Q],T) $ is a pair $(\K[Q]_t, T_t)$ where $\K[Q]_t$
is a $\K[[t]]$-algebra
given by $\K[Q]_t=\K[Q]  \otimes \K[[ t ]]$ with all ternary structures
inherited by extending those on $\K[Q]_t$
with the identity on the
$\K[[t]]$ factor (the trivial deformation as the algebra), with
a deformations of $T$ given by $T_t= T + t T_1 + \cdots +
t^n T_n + \cdots : \K[Q]_t \otimes \K[Q]_t \otimes \K[Q]_t \rightarrow \K[Q]_t$ where
$T_i: \K[Q] \otimes \K[Q] \otimes \K[Q] \rightarrow \K[Q] $,
 $i=1, 2, \cdots$, are linear
maps. } \end{definition}

\noindent
The map $T_t$ satisfies the equation
 \begin{eqnarray}\nonumber
T_t\circ (T_t \o id \o id)=T_t\circ (T_t \o T_t \o T_t)\circ\rho \circ (id\o id \o id \o D\o D).
\end{eqnarray}
That is for elements $x,y,z,u,v\in Q$, we have 
\begin{eqnarray}\label{equaDefo}
T_t(T_t(x \o y \o z) \o u \o v)=T_t(T_t(x \o u \o v) \o T_t(y \o u \o v) \o T_t(z \o u \o v))
\end{eqnarray}
We call the equation \eqref{equaDefo} the deformation equation of the ternary operation $T$.
 
\subsection{Deformation equation}
The deformation equation \eqref{equaDefo} may be written by  expanding and collecting the coefficients of $%
t^k$ as
\begin{eqnarray*}
 \sum_{k=0}^\infty t^{k}\sum_ { i=0}^k
T_i(T_{k-i}(x \o y \o z) \o u \o v)=  \sum_{k=0}^\infty  t^{k} \sum_ { m+n+p+q=k }     T_m(T_n(x \o u \o v) \o T_p(y \o u \o v) \o T_q(z \o u \o v)),
\end{eqnarray*}
where $m,n,p$ and $q$ are non-negative integers.
It  yields, for $k=0,1,2,\ldots$%
\begin{eqnarray*}
 \sum_ { i=0}^k
T_i(T_{k-i}(x \o y \o z) \o u \o v)=    \sum_ { m+n+p+q=k }     T_m(T_n(x \o u \o v) \o T_p(y \o u \o v) \o T_q(z \o u \o v))
.
\\
\end{eqnarray*}

\noindent
This infinite system gives the
necessary and sufficient conditions for $T _t$ to  be a distributive  ternary operation. The first problem is to give
conditions on $T_i$ so that the deformation $T _t$ is distributive.

\noindent
\\  The first equation $\left( k=0\right) $ is the ternary distributivity condition for $%
T _0.$\\The second equation $(k=1)$ shows that $T _1$ satisfies $\delta_m^2(T_1)=0$. More generally, suppose that $T _p$ be the first non-zero coefficient
after $T $ in the deformation $T _t$. This $T_p$ is called the {\it %
infinitesimal} of $T _t$ and should satisfy $\delta_m^2(T_p)=0$. In order to express it as a $2$-cocycle with respect to the cohomology defined above, we consider $(\K[Q],T) $ as a weak  bialgebra $(\K[Q],T,D) $ where $D$ is the diagonal map defined on the elements of $Q$ as $D(x)=x\o x\o x$. Therefore the pair $(T _p,0)$  is a 2-cocycle with respect to the cohomology of $(\K[Q],T,D) $. We call
ˆ $T_p$ a 2-cocycle of the ternary distributive    algebras
cohomology   with coefficient in itself.

\begin{theorem} The map $T _p$  is a 2-cocycle of the ternary distributive algebras
cohomology  with coefficient in itself.
\end{theorem}
\noindent
{\it Proof.} In the equation \eqref{equaDefo}, make the following substitution $k=p$ and $T
_1=\cdots =T _{p-1}=0$.
$\Box$

\subsubsection{Equivalent and trivial deformations}

We characterize the equivalent and trivial deformations
of  ternary distributive algebras.

\begin{definition}Let $(Q,T_0)$ be a ternary distributive set and   $A=(\K [Q],T_0) $ be a corresponding ternary distributive  algebra.
Let
 $T_{t}=\sum_{i\geq 0}T_{i}t^{i}$ and
$T'_{t}=\sum_{i\geq 0}T'_{i}t^{i}$ be two deformations of $(\K Q,T_0) $,
 $(T_{0}=T'_{0})$.
We say that they are equivalent if there exists a formal isomorphism
$\Phi_{t}: A\rightarrow A[[t]]$ which is a $\K[[t]]$-linear map
that may be written in the form
\begin{equation*}
\Phi_{t}=\sum_{i\geq 0}\Phi _{i}t^{i}
=id+\Phi _{1}t+\Phi_{2}t^{2}+\ldots
\qquad\text{where}\quad
\Phi_{i}\in End_{\K }(A)
\quad\text{and}\quad
\Phi_{0}=id,
\end{equation*}
such that
\begin{equation}
\label{equ2}
\Phi_{t}\circ T_{t}=T'_{t}\circ\Phi _{t}.
\end{equation}
A deformation $T_{t}$ of $T_{0}$ is said to be trivial if and only if $T_{t}$
is equivalent to $T_{0}.$
\end{definition}

The condition (\ref{equ2}) may be written
$
\Phi _{t}(T_{t}(x\otimes y\otimes z))
=T'_{t}(\Phi _{t}(x)\otimes \Phi _{t}(y)\otimes \Phi _{t}(z)),\ \forall x,y,z\in A,
$
which is equivalent to
\begin{equation}
\sum_{i\geq 0}\Phi _{i}\left(\sum_{j\geq 0}T_{j}
(x\otimes y\otimes z)t^{j}\right)t^{i}
=\sum_{i\geq 0}T'_{i}
\left(
\sum_{j\geq 0}\Phi _{j}(x)t^{j}\otimes\sum_{k\geq 0}\Phi_{k}(y)t^{k}
{\sum_{l\geq 0}}\Phi _{l}(z)t^{l}
\right)
t^{i},
\end{equation}
or
\begin{equation*}
\sum_{i,j\geq 0}\Phi _{i}(T_{j}(x\otimes y\otimes z))t^{i+j}
=
\sum_{i,j,k,l\geq 0}T'_{i}(\Phi _{j}(x)\otimes \Phi_{k}(y)\otimes \Phi_{l}(z))t^{i+j+k+l}.
\end{equation*}
By identification of  coefficients, one obtains that the constant
coefficients are identical
\begin{equation*}
T_{0}=T'_{0}
\quad\text{because}\quad
\Phi _{0}=id
\end{equation*}
and for  coefficients of $t$ one has
\begin{align*}
\Phi _{0}(T_{1}(x\otimes y\otimes z))+\Phi _{1}(T_{0}(x\otimes y\otimes z))
&=
T'_{1}(\Phi _{0}(x)\otimes \Phi _{0}(y)\otimes \Phi_{0}(z))+
T'_{0}(\Phi _{1}(x)\otimes \Phi _{0}(y)\otimes \Phi_{0}(z))\\
&+
T'_{0}(\Phi _{0}(x)\otimes \Phi _{1}(y)\otimes \Phi_{0}(z))+
T'_{0}(\Phi _{0}(x)\otimes \Phi _{0}(y)\otimes \Phi_{1}(z)).
\end{align*}
It follows
\begin{align*}
T_{1}(x\otimes y\otimes z)+\Phi _{1}(T_{0}(x\otimes y\otimes z))
&=
T'_{1}(x\otimes y\otimes z)+T_{0}(\Phi _{1}(x)\otimes
y\otimes z)\\
&+
T_{0}(x\otimes \Phi _{1}(y)\otimes z)+T_{0}(x\otimes y\otimes \Phi _{1}(z)).
\end{align*}
Consequently,
\begin{align*}
T'_{1}(x\otimes y\otimes z)
&=
T_{1}(x\otimes y\otimes z)+\Phi_{1}(T_{0}(x\otimes y\otimes z))-T_{0}(\Phi _{1}(x)\otimes y\otimes z)\\
&-
T_{0}(x\otimes \Phi _{1}(y)\otimes z)-T_{0}(x\otimes y\otimes \Phi _{1}(z)).
\end{align*}
That is $T'_{1}=T_{1}+\delta^1_m \Phi _{1}$.
Therefore $T_1$ and $T'_1$ are in the same cohomology class.

Thus, we have

\begin{theorem}
Let $(Q,T_0)$ be a ternary distributive set, $(A=\K[Q],T_{0})$ be a corresponding  ternary distributive algebra
and  $T_{t}$ be a one parameter family of deformations of $T_{0}$.
Then $T_{t}$ is equivalent to
\begin{equation}
T_{t}(x\otimes y\otimes z)
=T_{0}(x\otimes y\otimes z)+T_{p}^{\prime }(x\otimes y\otimes z)t^{p}
+T_{p+1}^{\prime }(x\otimes y\otimes z)t^{p+1}+\ldots
\end{equation}%
where $T_{p}^{\prime }\in \mathit{Z}^{2}(A,A)$ and
$T_{p}^{\prime }\notin \mathit{B}^{2}(A,A)$.

Moreover, if $\mathit{H}^{2}(A,A)=\{0\}$ then every deformation of $A$ is trivial. The ternary distributive algebra is said rigid.
\end{theorem}
\begin{proof}
Let $T_p$ be the first nonzero term in the deformation.
The deformation equation implies $\delta T_{p}=0$ which means
$T_{p}\in \mathit{Z}^{2}(A,A)$. If further $T_{p}\in
\mathit{B}^{2}(A,A)$, i.e. $T_{p}=\delta g$ with
$g\in Hom(A,A)$, then using a formal morphism $\Phi _{t}=id+t g$ we
obtain that the deformation $T_t$ is equivalent to the deformation
given for all $x,y,z\in A$ by
\begin{equation*}
T_{t}^{\prime }(x\otimes y\otimes z) =\Phi _{t}^{-1}\circ T_{t}\circ
(\Phi_{t}(x)\otimes \Phi _{t}(y)\otimes \Phi _{t}(z))
=T_{0}(x\otimes y\otimes z)+T^{\prime }_{p+1}(x\otimes y\otimes
z)t^{p+1}+\ldots
\end{equation*}
and again $\ T_{p+1}\in \mathit{Z}^{2}(T_{0},T_{0})$.
\end{proof}
We end the paper with the fact that we do not know yet how to use these ternary quandles to obtain invariant of knots and/or knotted surfaces. 

\subsection*{Acknowledgements} We wish to thank Edwin Clark for a reading the manuscript and also for a number of crucial discussions.

\appendix

\section{Ternary Distributivity from Coalgebras}\label{TDColag}

\noindent
In the following proposition we give a classification of ternary distributive linear maps $q:A \otimes A \otimes A \rightarrow A$ that are compatible with comultiplication meaning that $\Delta q=(q \otimes q \otimes q)\rho(\Delta \otimes \Delta \otimes \Delta)$.  Here $A=\K[Q]$ is a vector space of dimension two with a basis $Q=\{x,y\}$.  Here is an outline of the proof of the following proposition.  If we set  $q(x \otimes x \otimes x)=ax+by$, then we obtain that $a\Delta(x)+b\Delta(y)=(ax+by)\otimes  (ax+by)\otimes (ax+by) $ which implies that $a^3=a,\; b^3=b, a^2b=0$ and $ab^2=0$.   Thus the only possible values of $q(x \otimes x \otimes x)$ are $0$, $\pm x$ or $\pm y$.  The same holds for the other generators $x \otimes x \otimes y$ etc.\\
With these notations we have the following

  \begin{proposition}
A linear map $q:A \otimes A \otimes A \rightarrow A$ satisfies equation (\ref{C1-ternary}) (that is ternary distributive)
and compatible with the comultiplication if and only if $q$ is
one of the  functions indicated via any column in the table below. The values are determined on the basis elements
$ x \otimes x \otimes x,$ through $y\otimes y \otimes y$ as indicated in the following chart


\scriptsize
\begin{flushleft}
\begin{tabular}{| l | c | c | c | c | c | c | c | c | c | c | c | c | c | c |  }
\hline
$q(x,x,x)$ & $    0$ & $    0$ & $    0$ & $    0$ & $    0$ & $    0$ & $    0$ & $    0$ & $    0$ & $    0$ & $    0$ & $    0$ & $    0$ & $    0$ \\
$q(x,x,y)$ & $    0$ & $    0$ & $    0$ & $    0$ & $    0$ & $    0$ & $    0$ & $    0$ & $    0$ & $    0$ & $    0$ & $    0$ & $    0$ & $    0$ \\
$q(x,y,x)$ & $    0$ & $    0$ & $    0$ & $    0$ & $    0$ & $    0$ & $    0$ & $    0$ & $    0$ & $    0$ & $    0$ & $    0$ & $    0$ & $    0$ \\
$q(x,y,y)$ & $    0$ & $    0$ & $    0$ & $    0$ & $    0$ & $    0$ & $    0$ & $    0$ & $    0$ & $    0$ & $    0$ & $    0$ & $    0$ & $    0$ \\
$q(y,x,x)$ & $    0$ & $    0$ & $    0$ & $    0$ & $    0$ & $    0$ & $    0$ & $    0$ & $    0$ & $    0$ & $    0$ & $    0$ & $    0$ & $    0$ \\
$q(y,x,y)$ & $    0$ & $    0$ & $    0$ & $    0$ & $    0$ & $    0$ & $\pm x$ & $\pm x$ & $\pm x$ & $\pm x$ & $\pm x$ & $\pm x$ & $\pm x$ & $\pm x$ \\
$q(y,y,x)$ & $    0$ & $    0$ & $    0$ & $\pm x$ & $\pm x$ & $\pm x$ & $\mp x$ & $\mp x$ & $\mp x$ & $    0$ & $    0$ & $    0$ & $\pm x$ & $\pm x$ \\
$q(y,y,y)$ & $    0$ & $\pm x$ & $\pm y$ & $\mp x$ & $    0$ & $\pm x$ & $\mp x$ & $    0$ & $\pm x$ & $\mp x$ & $    0$ & $\pm x$ & $\mp x$ & $    0$ \\
\hline
\end{tabular}
\begin{tabular}{| l | c | c | c | c | c | c | c | c | c | c | c | c | c | c |  }
\hline
$q(x,x,x)$ & $    0$ & $    0$ & $    0$ & $    0$ & $    0$ & $    0$ & $    0$ & $    0$ & $    0$ & $    0$ & $    0$ & $    0$ & $    0$ & $    0$ \\
$q(x,x,y)$ & $    0$ & $    0$ & $    0$ & $    0$ & $    0$ & $    0$ & $    0$ & $    0$ & $    0$ & $    0$ & $    0$ & $    0$ & $    0$ & $    0$ \\
$q(x,y,x)$ & $    0$ & $    0$ & $    0$ & $    0$ & $    0$ & $    0$ & $    0$ & $    0$ & $    0$ & $    0$ & $    0$ & $    0$ & $    0$ & $    0$ \\
$q(x,y,y)$ & $    0$ & $    0$ & $    0$ & $    0$ & $    0$ & $    0$ & $    0$ & $    0$ & $    0$ & $    0$ & $    0$ & $    0$ & $    0$ & $    0$ \\
$q(y,x,x)$ & $    0$ & $\pm x$ & $\pm x$ & $\pm x$ & $\pm x$ & $\pm x$ & $\pm x$ & $\pm x$ & $\pm x$ & $\pm x$ & $\pm x$ & $\pm x$ & $\pm x$ & $\pm x$ \\
$q(y,x,y)$ & $\pm x$ & $\mp x$ & $\mp x$ & $\mp x$ & $\mp x$ & $\mp x$ & $\mp x$ & $\mp x$ & $\mp x$ & $\mp x$ & $    0$ & $    0$ & $    0$ & $    0$ \\
$q(y,y,x)$ & $\pm x$ & $\mp x$ & $\mp x$ & $\mp x$ & $    0$ & $    0$ & $    0$ & $\pm x$ & $\pm x$ & $\pm x$ & $\mp x$ & $\mp x$ & $\mp x$ & $    0$ \\
$q(y,y,y)$ & $\pm x$ & $\mp x$ & $    0$ & $\pm x$ & $\mp x$ & $    0$ & $\pm x$ & $\mp x$ & $    0$ & $\pm x$ & $\mp x$ & $    0$ & $\pm x$ & $\mp x$ \\
\hline
\end{tabular}
\begin{tabular}{| l | c | c | c | c | c | c | c | c | c | c | c | c | c | c |  }
\hline
$q(x,x,x)$ & $    0$ & $    0$ & $    0$ & $    0$ & $    0$ & $    0$ & $    0$ & $    0$ & $    0$ & $    0$ & $    0$ & $    0$ & $    0$ & $    0$ \\
$q(x,x,y)$ & $    0$ & $    0$ & $    0$ & $    0$ & $    0$ & $    0$ & $    0$ & $    0$ & $    0$ & $    0$ & $    0$ & $    0$ & $    0$ & $    0$ \\
$q(x,y,x)$ & $    0$ & $    0$ & $    0$ & $    0$ & $    0$ & $    0$ & $    0$ & $    0$ & $    0$ & $    0$ & $    0$ & $    0$ & $    0$ & $    0$ \\
$q(x,y,y)$ & $    0$ & $    0$ & $    0$ & $    0$ & $    0$ & $    0$ & $    0$ & $    0$ & $    0$ & $    0$ & $    0$ & $    0$ & $    0$ & $    0$ \\
$q(y,x,x)$ & $\pm x$ & $\pm x$ & $\pm x$ & $\pm x$ & $\pm x$ & $\pm x$ & $\pm x$ & $\pm x$ & $\pm x$ & $\pm x$ & $\pm x$ & $\pm x$ & $\pm x$ & $\pm x$ \\
$q(y,x,y)$ & $    0$ & $    0$ & $    0$ & $    0$ & $    0$ & $\pm x$ & $\pm x$ & $\pm x$ & $\pm x$ & $\pm x$ & $\pm x$ & $\pm x$ & $\pm x$ & $\pm x$ \\
$q(y,y,x)$ & $    0$ & $    0$ & $\pm x$ & $\pm x$ & $\pm x$ & $\mp x$ & $\mp x$ & $\mp x$ & $    0$ & $    0$ & $    0$ & $\pm x$ & $\pm x$ & $\pm x$ \\
$q(y,y,y)$ & $    0$ & $\pm x$ & $\mp x$ & $    0$ & $\pm x$ & $\mp x$ & $    0$ & $\pm x$ & $\mp x$ & $    0$ & $\pm x$ & $\mp x$ & $    0$ & $\pm x$ \\
\hline
\end{tabular}
\begin{tabular}{| l | c | c | c | c | c | c | c | c | c | c | c | c | c | c |  }
\hline
$q(x,x,x)$ & $    0$ & $    0$ & $    0$ & $    0$ & $    0$ & $    0$ & $    0$ & $    0$ & $    0$ & $    0$ & $    0$ & $    0$ & $    0$ & $    0$ \\
$q(x,x,y)$ & $    0$ & $    0$ & $    0$ & $    0$ & $    0$ & $    0$ & $    0$ & $    0$ & $    0$ & $\pm x$ & $\pm x$ & $\pm x$ & $\pm x$ & $\pm x$ \\
$q(x,y,x)$ & $    0$ & $    0$ & $    0$ & $\pm x$ & $\pm x$ & $\pm x$ & $\pm y$ & $\pm y$ & $\pm y$ & $\mp x$ & $\mp x$ & $\mp x$ & $    0$ & $    0$ \\
$q(x,y,y)$ & $\pm x$ & $\pm x$ & $\pm y$ & $\mp x$ & $    0$ & $\pm x$ & $\mp y$ & $    0$ & $\pm y$ & $\mp x$ & $    0$ & $\pm x$ & $\mp x$ & $    0$ \\
$q(y,x,x)$ & $    0$ & $    0$ & $    0$ & $    0$ & $    0$ & $    0$ & $    0$ & $    0$ & $    0$ & $    0$ & $    0$ & $    0$ & $    0$ & $    0$ \\
$q(y,x,y)$ & $    0$ & $    0$ & $    0$ & $    0$ & $    0$ & $    0$ & $    0$ & $    0$ & $    0$ & $\pm y$ & $\pm y$ & $\pm y$ & $\pm y$ & $\pm y$ \\
$q(y,y,x)$ & $    0$ & $    0$ & $    0$ & $\pm y$ & $\pm y$ & $\pm y$ & $    0$ & $    0$ & $    0$ & $\mp y$ & $\mp y$ & $\mp y$ & $    0$ & $    0$ \\
$q(y,y,y)$ & $\mp y$ & $\pm y$ & $    0$ & $\mp y$ & $    0$ & $\pm y$ & $    0$ & $    0$ & $    0$ & $\mp y$ & $    0$ & $\pm y$ & $\mp y$ & $    0$ \\
\hline
\end{tabular}
\begin{tabular}{| l | c | c | c | c | c | c | c | c | c | c | c | c | c | c |  }
\hline
$q(x,x,x)$ & $    0$ & $    0$ & $    0$ & $    0$ & $    0$ & $    0$ & $    0$ & $    0$ & $    0$ & $    0$ & $    0$ & $    0$ & $    0$ & $    0$ \\
$q(x,x,y)$ & $\pm x$ & $\pm x$ & $\pm x$ & $\pm x$ & $\pm y$ & $\pm y$ & $\pm y$ & $\pm y$ & $\pm y$ & $\pm y$ & $\pm y$ & $\pm y$ & $\pm y$ & $\pm y$ \\
$q(x,y,x)$ & $    0$ & $\pm x$ & $\pm x$ & $\pm x$ & $\mp y$ & $\mp y$ & $\mp y$ & $\mp y$ & $    0$ & $    0$ & $    0$ & $\pm y$ & $\pm y$ & $\pm y$ \\
$q(x,y,y)$ & $\pm x$ & $\mp x$ & $    0$ & $\pm x$ & $\mp y$ & $    0$ & $    0$ & $\pm y$ & $\mp y$ & $    0$ & $\pm y$ & $\mp y$ & $    0$ & $    0$ \\
$q(y,x,x)$ & $    0$ & $    0$ & $    0$ & $    0$ & $    0$ & $    0$ & $    0$ & $    0$ & $    0$ & $    0$ & $    0$ & $    0$ & $    0$ & $    0$ \\
$q(y,x,y)$ & $\pm y$ & $\pm y$ & $\pm y$ & $\pm y$ & $    0$ & $\mp x$ & $    0$ & $    0$ & $    0$ & $    0$ & $    0$ & $    0$ & $    0$ & $\pm x$ \\
$q(y,y,x)$ & $    0$ & $\pm y$ & $\pm y$ & $\pm y$ & $    0$ & $\pm x$ & $    0$ & $    0$ & $    0$ & $    0$ & $    0$ & $    0$ & $    0$ & $\pm x$ \\
$q(y,y,y)$ & $\pm y$ & $\mp y$ & $    0$ & $\pm y$ & $    0$ & $    0$ & $    0$ & $    0$ & $    0$ & $    0$ & $    0$ & $    0$ & $    0$ & $    0$ \\
\hline
\end{tabular}
\begin{tabular}{| l | c | c | c | c | c | c | c | c | c | c | c | c | c | c |  }
\hline
$q(x,x,x)$ & $    0$ & $\pm x$ & $\pm x$ & $\pm x$ & $\pm x$ & $\pm x$ & $\pm x$ & $\pm x$ & $\pm x$ & $\pm x$ & $\pm x$ & $\pm x$ & $\pm x$ & $\pm x$ \\
$q(x,x,y)$ & $\pm y$ & $\mp y$ & $\mp y$ & $\mp y$ & $\mp y$ & $\mp x$ & $\mp x$ & $\mp x$ & $\mp x$ & $\mp x$ & $\mp x$ & $\mp x$ & $\mp x$ & $\mp x$ \\
$q(x,y,x)$ & $\pm y$ & $\mp y$ & $\mp y$ & $\pm y$ & $\pm y$ & $\mp x$ & $\mp x$ & $\mp x$ & $\mp x$ & $\mp x$ & $\mp x$ & $\mp x$ & $\mp x$ & $\mp x$ \\
$q(x,y,y)$ & $\pm y$ & $\mp x$ & $\pm x$ & $\mp x$ & $\pm x$ & $\mp y$ & $\mp y$ & $\mp y$ & $\mp y$ & $\mp y$ & $\mp y$ & $\mp y$ & $\mp y$ & $\mp x$ \\
$q(y,x,x)$ & $    0$ & $\mp y$ & $\pm y$ & $\mp y$ & $\pm y$ & $\mp x$ & $\mp x$ & $\mp x$ & $\mp x$ & $\pm y$ & $\pm y$ & $\pm y$ & $\pm y$ & $\pm y$ \\
$q(y,x,y)$ & $    0$ & $\mp x$ & $\mp x$ & $\pm x$ & $\pm x$ & $\mp y$ & $\mp y$ & $\pm x$ & $\pm x$ & $\mp y$ & $\mp y$ & $\pm x$ & $\pm x$ & $\mp y$ \\
$q(y,y,x)$ & $    0$ & $\mp x$ & $\mp x$ & $\mp x$ & $\mp x$ & $\mp y$ & $\pm x$ & $\mp y$ & $\pm x$ & $\mp y$ & $\pm x$ & $\mp y$ & $\pm x$ & $\mp y$ \\
$q(y,y,y)$ & $    0$ & $\pm y$ & $\pm y$ & $\pm y$ & $\pm y$ & $\pm y$ & $\pm y$ & $\pm y$ & $\pm y$ & $\pm y$ & $\pm y$ & $\pm y$ & $\pm y$ & $\mp y$ \\
\hline
\end{tabular}
\begin{tabular}{| l | c | c | c | c | c | c | c | c | c | c | c | c | c | c |  }
\hline
$q(x,x,x)$ & $\pm x$ & $\pm x$ & $\pm x$ & $\pm x$ & $\pm x$ & $\pm x$ & $\pm x$ & $\pm x$ & $\pm x$ & $\pm x$ & $\pm x$ & $\pm x$ & $\pm x$ & $\pm x$ \\
$q(x,x,y)$ & $\mp x$ & $\mp x$ & $\mp x$ & $\mp x$ & $\mp x$ & $\mp x$ & $\mp x$ & $\mp x$ & $\mp x$ & $\mp x$ & $\mp x$ & $\mp x$ & $\mp x$ & $\mp x$ \\
$q(x,y,x)$ & $\mp x$ & $\mp x$ & $\mp x$ & $\mp x$ & $\mp x$ & $\mp x$ & $\mp x$ & $\mp x$ & $\mp x$ & $\mp x$ & $\mp x$ & $\mp x$ & $\mp x$ & $\mp x$ \\
$q(x,y,y)$ & $    0$ & $\pm x$ & $\pm x$ & $\pm x$ & $\pm x$ & $\pm x$ & $\pm x$ & $\pm x$ & $\pm x$ & $\pm x$ & $\pm x$ & $\pm x$ & $\pm x$ & $\pm x$ \\
$q(y,x,x)$ & $\pm y$ & $\mp x$ & $\mp x$ & $\mp x$ & $\mp x$ & $\mp x$ & $\mp x$ & $\mp x$ & $\mp x$ & $\pm y$ & $\pm y$ & $\pm y$ & $\pm y$ & $\pm y$ \\
$q(y,x,y)$ & $\mp y$ & $\mp y$ & $\mp y$ & $\mp y$ & $\mp y$ & $\pm x$ & $\pm x$ & $\pm x$ & $\pm x$ & $\mp y$ & $\mp y$ & $\mp y$ & $\mp y$ & $\pm x$ \\
$q(y,y,x)$ & $\mp y$ & $\mp y$ & $\mp y$ & $\pm x$ & $\pm x$ & $\mp y$ & $\mp y$ & $\pm x$ & $\pm x$ & $\mp y$ & $\mp y$ & $\pm x$ & $\pm x$ & $\mp y$ \\
$q(y,y,y)$ & $    0$ & $\mp x$ & $\pm y$ & $\mp x$ & $\pm y$ & $\mp x$ & $\pm y$ & $\mp x$ & $\pm y$ & $\mp x$ & $\pm y$ & $\mp x$ & $\pm y$ & $\mp x$ \\
\hline
\end{tabular}
\begin{tabular}{| l | c | c | c | c | c | c | c | c | c | c | c | c | c | c |  }
\hline
$q(x,x,x)$ & $\pm x$ & $\pm x$ & $\pm x$ & $\pm x$ & $\pm x$ & $\pm x$ & $\pm x$ & $\pm x$ & $\pm x$ & $\pm x$ & $\pm x$ & $\pm x$ & $\pm x$ & $\pm x$ \\
$q(x,x,y)$ & $\mp x$ & $\mp x$ & $\mp x$ & $\mp x$ & $\mp x$ & $\mp x$ & $\mp x$ & $\mp x$ & $\mp x$ & $\mp x$ & $\mp x$ & $\mp x$ & $\mp x$ & $\mp x$ \\
$q(x,y,x)$ & $\mp x$ & $\mp x$ & $\mp x$ & $    0$ & $    0$ & $    0$ & $\pm x$ & $\pm x$ & $\pm x$ & $\pm y$ & $\pm y$ & $\pm y$ & $\pm y$ & $\pm y$ \\
$q(x,y,y)$ & $\pm x$ & $\pm x$ & $\pm x$ & $\mp x$ & $    0$ & $\pm x$ & $\mp x$ & $    0$ & $\pm x$ & $\mp y$ & $\mp y$ & $\mp y$ & $\mp y$ & $\pm x$ \\
$q(y,x,x)$ & $\pm y$ & $\pm y$ & $\pm y$ & $\pm y$ & $\pm y$ & $\pm y$ & $\pm y$ & $\pm y$ & $\pm y$ & $\mp x$ & $\mp x$ & $\pm y$ & $\pm y$ & $\mp x$ \\
$q(y,x,y)$ & $\pm x$ & $\pm x$ & $\pm x$ & $\mp y$ & $\mp y$ & $\mp y$ & $\mp y$ & $\mp y$ & $\mp y$ & $\mp y$ & $\pm x$ & $\mp y$ & $\pm x$ & $\mp y$ \\
$q(y,y,x)$ & $\mp y$ & $\pm x$ & $\pm x$ & $    0$ & $    0$ & $    0$ & $\pm y$ & $\pm y$ & $\pm y$ & $\mp y$ & $\mp y$ & $\mp y$ & $\mp y$ & $\mp y$ \\
$q(y,y,y)$ & $\pm y$ & $\mp x$ & $\pm y$ & $\mp y$ & $    0$ & $\pm y$ & $\mp y$ & $    0$ & $\pm y$ & $\pm y$ & $\pm y$ & $\pm y$ & $\pm y$ & $\pm y$ \\
\hline
\end{tabular}
\begin{tabular}{| l | c | c | c | c | c | c | c | c | c | c | c | c | c | c |  }
\hline
$q(x,x,x)$ & $\pm x$ & $\pm x$ & $\pm x$ & $\pm x$ & $\pm x$ & $\pm x$ & $\pm x$ & $\pm x$ & $\pm x$ & $\pm x$ & $\pm x$ & $\pm x$ & $\pm x$ & $\pm x$ \\
$q(x,x,y)$ & $\mp x$ & $\mp x$ & $\mp x$ & $    0$ & $    0$ & $    0$ & $    0$ & $    0$ & $    0$ & $    0$ & $    0$ & $    0$ & $    0$ & $    0$ \\
$q(x,y,x)$ & $\pm y$ & $\pm y$ & $\pm y$ & $\mp x$ & $\mp x$ & $\mp x$ & $    0$ & $    0$ & $    0$ & $    0$ & $    0$ & $    0$ & $    0$ & $    0$ \\
$q(x,y,y)$ & $\pm x$ & $\pm x$ & $\pm x$ & $\mp x$ & $    0$ & $\pm x$ & $\mp x$ & $\mp x$ & $\mp x$ & $\mp x$ & $    0$ & $    0$ & $    0$ & $    0$ \\
$q(y,x,x)$ & $\mp x$ & $\pm y$ & $\pm y$ & $\pm y$ & $\pm y$ & $\pm y$ & $\mp y$ & $\mp y$ & $\pm y$ & $\pm y$ & $\mp y$ & $    0$ & $    0$ & $    0$ \\
$q(y,x,y)$ & $\pm x$ & $\mp y$ & $\pm x$ & $    0$ & $    0$ & $    0$ & $    0$ & $    0$ & $    0$ & $    0$ & $    0$ & $    0$ & $    0$ & $    0$ \\
$q(y,y,x)$ & $\mp y$ & $\mp y$ & $\mp y$ & $\mp y$ & $\mp y$ & $\mp y$ & $    0$ & $    0$ & $    0$ & $    0$ & $    0$ & $    0$ & $    0$ & $    0$ \\
$q(y,y,y)$ & $\pm y$ & $\pm y$ & $\pm y$ & $\mp y$ & $    0$ & $\pm y$ & $\mp y$ & $\pm y$ & $\mp y$ & $\pm y$ & $    0$ & $\mp y$ & $    0$ & $\pm y$ \\
\hline
\end{tabular}
\begin{tabular}{| l | c | c | c | c | c | c | c | c | c | c | c | c | c | c |  }
\hline
$q(x,x,x)$ & $\pm x$ & $\pm x$ & $\pm x$ & $\pm x$ & $\pm x$ & $\pm x$ & $\pm x$ & $\pm x$ & $\pm x$ & $\pm x$ & $\pm x$ & $\pm x$ & $\pm x$ & $\pm x$ \\
$q(x,x,y)$ & $    0$ & $    0$ & $    0$ & $    0$ & $    0$ & $    0$ & $    0$ & $    0$ & $\pm x$ & $\pm x$ & $\pm x$ & $\pm x$ & $\pm x$ & $\pm x$ \\
$q(x,y,x)$ & $    0$ & $    0$ & $    0$ & $    0$ & $    0$ & $\pm x$ & $\pm x$ & $\pm x$ & $\mp x$ & $\mp x$ & $\mp x$ & $    0$ & $    0$ & $    0$ \\
$q(x,y,y)$ & $    0$ & $\pm x$ & $\pm x$ & $\pm x$ & $\pm x$ & $\mp x$ & $    0$ & $\pm x$ & $\mp x$ & $    0$ & $\pm x$ & $\mp x$ & $    0$ & $\pm x$ \\
$q(y,x,x)$ & $\pm y$ & $\mp y$ & $\mp y$ & $\pm y$ & $\pm y$ & $\pm y$ & $\pm y$ & $\pm y$ & $\pm y$ & $\pm y$ & $\pm y$ & $\pm y$ & $\pm y$ & $\pm y$ \\
$q(y,x,y)$ & $    0$ & $    0$ & $    0$ & $    0$ & $    0$ & $    0$ & $    0$ & $    0$ & $\pm y$ & $\pm y$ & $\pm y$ & $\pm y$ & $\pm y$ & $\pm y$ \\
$q(y,y,x)$ & $    0$ & $    0$ & $    0$ & $    0$ & $    0$ & $\pm y$ & $\pm y$ & $\pm y$ & $\mp y$ & $\mp y$ & $\mp y$ & $    0$ & $    0$ & $    0$ \\
$q(y,y,y)$ & $    0$ & $\mp y$ & $\pm y$ & $\mp y$ & $\pm y$ & $\mp y$ & $    0$ & $\pm y$ & $\mp y$ & $    0$ & $\pm y$ & $\mp y$ & $    0$ & $\pm y$ \\
\hline
\end{tabular}
\begin{tabular}{| l | c | c | c | c | c | c | c | c | c | c | c | c | c | c |  }
\hline
$q(x,x,x)$ & $\pm x$ & $\pm x$ & $\pm x$ & $\pm x$ & $\pm x$ & $\pm x$ & $\pm x$ & $\pm x$ & $\pm x$ & $\pm x$ & $\pm x$ & $\pm x$ & $\pm x$ & $\pm x$ \\
$q(x,x,y)$ & $\pm x$ & $\pm x$ & $\pm x$ & $\pm x$ & $\pm x$ & $\pm x$ & $\pm x$ & $\pm x$ & $\pm x$ & $\pm x$ & $\pm x$ & $\pm x$ & $\pm x$ & $\pm x$ \\
$q(x,y,x)$ & $\pm x$ & $\pm x$ & $\pm x$ & $\pm x$ & $\pm x$ & $\pm x$ & $\pm x$ & $\pm x$ & $\pm x$ & $\pm x$ & $\pm x$ & $\pm x$ & $\pm x$ & $\pm x$ \\
$q(x,y,y)$ & $\mp x$ & $    0$ & $\pm x$ & $\pm x$ & $\pm x$ & $\pm x$ & $\pm x$ & $\pm x$ & $\pm x$ & $\pm x$ & $\pm x$ & $\pm x$ & $\pm x$ & $\pm x$ \\
$q(y,x,x)$ & $\pm y$ & $\pm y$ & $\pm x$ & $\pm x$ & $\pm x$ & $\pm x$ & $\pm x$ & $\pm x$ & $\pm x$ & $\pm x$ & $\pm y$ & $\pm y$ & $\pm y$ & $\pm y$ \\
$q(y,x,y)$ & $\pm y$ & $\pm y$ & $\pm x$ & $\pm x$ & $\pm x$ & $\pm x$ & $\pm y$ & $\pm y$ & $\pm y$ & $\pm y$ & $\pm x$ & $\pm x$ & $\pm x$ & $\pm x$ \\
$q(y,y,x)$ & $\pm y$ & $\pm y$ & $\pm x$ & $\pm x$ & $\pm y$ & $\pm y$ & $\pm x$ & $\pm x$ & $\pm y$ & $\pm y$ & $\pm x$ & $\pm x$ & $\pm y$ & $\pm y$ \\
$q(y,y,y)$ & $\mp y$ & $    0$ & $\pm x$ & $\pm y$ & $\pm x$ & $\pm y$ & $\pm x$ & $\pm y$ & $\pm x$ & $\pm y$ & $\pm x$ & $\pm y$ & $\pm x$ & $\pm y$ \\
\hline
\end{tabular}
\begin{tabular}{| l | c | c | c | c | c | c | c | c | c | c | c | c | c | c |  }
\hline
$q(x,x,x)$ & $\pm x$ & $\pm x$ & $\pm x$ & $\pm x$ & $\pm x$ & $\pm x$ & $\pm x$ & $\pm x$ & $\pm x$ & $\pm x$ & $\pm x$ & $\pm x$ & $\pm x$ & $\pm x$ \\
$q(x,x,y)$ & $\pm x$ & $\pm x$ & $\pm x$ & $\pm x$ & $\pm x$ & $\pm x$ & $\pm x$ & $\pm x$ & $\pm x$ & $\pm x$ & $\pm x$ & $\pm x$ & $\pm x$ & $\pm x$ \\
$q(x,y,x)$ & $\pm x$ & $\pm x$ & $\pm x$ & $\pm x$ & $\pm x$ & $\pm x$ & $\pm x$ & $\pm x$ & $\pm x$ & $\pm x$ & $\pm x$ & $\pm x$ & $\pm y$ & $\pm y$ \\
$q(x,y,y)$ & $\pm x$ & $\pm x$ & $\pm x$ & $\pm x$ & $\pm y$ & $\pm y$ & $\pm y$ & $\pm y$ & $\pm y$ & $\pm y$ & $\pm y$ & $\pm y$ & $\pm x$ & $\pm x$ \\
$q(y,x,x)$ & $\pm y$ & $\pm y$ & $\pm y$ & $\pm y$ & $\pm x$ & $\pm x$ & $\pm x$ & $\pm x$ & $\pm y$ & $\pm y$ & $\pm y$ & $\pm y$ & $\pm x$ & $\pm x$ \\
$q(y,x,y)$ & $\pm y$ & $\pm y$ & $\pm y$ & $\pm y$ & $\pm x$ & $\pm x$ & $\pm y$ & $\pm y$ & $\pm x$ & $\pm x$ & $\pm y$ & $\pm y$ & $\pm x$ & $\pm y$ \\
$q(y,y,x)$ & $\pm x$ & $\pm x$ & $\pm y$ & $\pm y$ & $\pm x$ & $\pm y$ & $\pm x$ & $\pm y$ & $\pm x$ & $\pm y$ & $\pm x$ & $\pm y$ & $\pm y$ & $\pm y$ \\
$q(y,y,y)$ & $\pm x$ & $\pm y$ & $\pm x$ & $\pm y$ & $\pm y$ & $\pm y$ & $\pm y$ & $\pm y$ & $\pm y$ & $\pm y$ & $\pm y$ & $\pm y$ & $\pm y$ & $\pm y$ \\
\hline
\end{tabular}
\begin{tabular}{| l | c | c | c | c | c | c | c | c | c | c | c | c | c | c |  }
\hline
$q(x,x,x)$ & $\pm x$ & $\pm x$ & $\pm x$ & $\pm x$ & $\pm x$ & $\pm x$ & $\pm x$ & $\pm x$ & $\pm x$ & $\pm x$ & $\pm x$ & $\pm x$ & $\pm x$ & $\pm x$ \\
$q(x,x,y)$ & $\pm x$ & $\pm x$ & $\pm x$ & $\pm x$ & $\pm x$ & $\pm x$ & $\pm y$ & $\pm y$ & $\pm y$ & $\pm y$ & $\pm y$ & $\pm y$ & $\pm y$ & $\pm y$ \\
$q(x,y,x)$ & $\pm y$ & $\pm y$ & $\pm y$ & $\pm y$ & $\pm y$ & $\pm y$ & $\mp y$ & $\mp y$ & $\mp x$ & $\mp x$ & $\mp x$ & $\mp x$ & $\mp x$ & $\mp x$ \\
$q(x,y,y)$ & $\pm x$ & $\pm x$ & $\pm y$ & $\pm y$ & $\pm y$ & $\pm y$ & $\mp x$ & $\pm x$ & $\mp y$ & $\mp y$ & $\mp y$ & $\mp y$ & $\pm x$ & $\pm x$ \\
$q(y,x,x)$ & $\pm y$ & $\pm y$ & $\pm x$ & $\pm x$ & $\pm y$ & $\pm y$ & $\mp y$ & $\pm y$ & $\mp x$ & $\mp x$ & $\pm y$ & $\pm y$ & $\mp x$ & $\mp x$ \\
$q(y,x,y)$ & $\pm x$ & $\pm y$ & $\pm x$ & $\pm y$ & $\pm x$ & $\pm y$ & $\mp x$ & $\mp x$ & $\mp y$ & $\mp y$ & $\mp y$ & $\mp y$ & $\mp y$ & $\mp y$ \\
$q(y,y,x)$ & $\pm y$ & $\pm y$ & $\pm y$ & $\pm y$ & $\pm y$ & $\pm y$ & $\pm x$ & $\pm x$ & $\mp y$ & $\pm x$ & $\mp y$ & $\pm x$ & $\mp y$ & $\pm x$ \\
$q(y,y,y)$ & $\pm y$ & $\pm y$ & $\pm y$ & $\pm y$ & $\pm y$ & $\pm y$ & $\pm y$ & $\pm y$ & $\pm y$ & $\pm y$ & $\pm y$ & $\pm y$ & $\pm y$ & $\pm y$ \\
\hline
\end{tabular}
\begin{tabular}{| l | c | c | c | c | c | c | c | c | c | c | c | c | c | c |  }
\hline
$q(x,x,x)$ & $\pm x$ & $\pm x$ & $\pm x$ & $\pm x$ & $\pm x$ & $\pm x$ & $\pm x$ & $\pm x$ & $\pm x$ & $\pm x$ & $\pm x$ & $\pm x$ & $\pm x$ & $\pm x$ \\
$q(x,x,y)$ & $\pm y$ & $\pm y$ & $\pm y$ & $\pm y$ & $\pm y$ & $\pm y$ & $\pm y$ & $\pm y$ & $\pm y$ & $\pm y$ & $\pm y$ & $\pm y$ & $\pm y$ & $\pm y$ \\
$q(x,y,x)$ & $\mp x$ & $\mp x$ & $\pm x$ & $\pm x$ & $\pm x$ & $\pm x$ & $\pm x$ & $\pm x$ & $\pm x$ & $\pm x$ & $\pm y$ & $\pm y$ & $\pm y$ & $\pm y$ \\
$q(x,y,y)$ & $\pm x$ & $\pm x$ & $\pm x$ & $\pm x$ & $\pm x$ & $\pm x$ & $\pm y$ & $\pm y$ & $\pm y$ & $\pm y$ & $\mp y$ & $\mp y$ & $\mp y$ & $\mp x$ \\
$q(y,x,x)$ & $\pm y$ & $\pm y$ & $\pm x$ & $\pm x$ & $\pm y$ & $\pm y$ & $\pm x$ & $\pm x$ & $\pm y$ & $\pm y$ & $\mp x$ & $\mp x$ & $\pm y$ & $\mp y$ \\
$q(y,x,y)$ & $\mp y$ & $\mp y$ & $\pm y$ & $\pm y$ & $\pm y$ & $\pm y$ & $\pm y$ & $\pm y$ & $\pm y$ & $\pm y$ & $\mp y$ & $\pm x$ & $\mp y$ & $\pm x$ \\
$q(y,y,x)$ & $\mp y$ & $\pm x$ & $\pm x$ & $\pm y$ & $\pm x$ & $\pm y$ & $\pm x$ & $\pm y$ & $\pm x$ & $\pm y$ & $\mp y$ & $\pm x$ & $\mp y$ & $\pm x$ \\
$q(y,y,y)$ & $\pm y$ & $\pm y$ & $\pm y$ & $\pm y$ & $\pm y$ & $\pm y$ & $\pm y$ & $\pm y$ & $\pm y$ & $\pm y$ & $\pm y$ & $\pm y$ & $\pm y$ & $\pm y$ \\
\hline
\end{tabular}
\begin{tabular}{| l | c | c | c | c | c | c | c | c | c | c | c | c | c | c |  }
\hline
$q(x,x,x)$ & $\pm x$ & $\pm x$ & $\pm x$ & $\pm x$ & $\pm x$ & $\pm x$ & $\pm x$ & $\pm x$ & $\pm y$ & $\pm y$ & $\pm y$ & $\pm y$ & $\pm y$ & $\pm y$ \\
$q(x,x,y)$ & $\pm y$ & $\pm y$ & $\pm y$ & $\pm y$ & $\pm y$ & $\pm y$ & $\pm y$ & $\pm y$ & $\mp y$ & $\mp y$ & $\mp y$ & $\mp y$ & $\mp y$ & $\mp y$ \\
$q(x,y,x)$ & $\pm y$ & $\pm y$ & $\pm y$ & $\pm y$ & $\pm y$ & $\pm y$ & $\pm y$ & $\pm y$ & $\mp y$ & $\mp y$ & $\mp y$ & $\mp y$ & $\mp y$ & $\mp y$ \\
$q(x,y,y)$ & $\pm x$ & $\pm x$ & $\pm x$ & $\pm x$ & $\pm x$ & $\pm y$ & $\pm y$ & $\pm y$ & $\mp y$ & $\mp x$ & $    0$ & $\pm y$ & $\pm y$ & $\pm y$ \\
$q(y,x,x)$ & $\mp x$ & $\pm x$ & $\pm y$ & $\pm y$ & $\pm y$ & $\pm x$ & $\pm x$ & $\pm y$ & $    0$ & $\mp y$ & $    0$ & $\mp y$ & $    0$ & $\pm x$ \\
$q(y,x,y)$ & $\mp y$ & $\pm y$ & $\mp y$ & $\pm x$ & $\pm y$ & $\pm x$ & $\pm y$ & $\pm y$ & $    0$ & $\pm y$ & $    0$ & $\pm y$ & $    0$ & $\mp x$ \\
$q(y,y,x)$ & $\mp y$ & $\pm y$ & $\mp y$ & $\pm x$ & $\pm y$ & $\pm x$ & $\pm y$ & $\pm y$ & $    0$ & $\pm y$ & $    0$ & $\pm y$ & $    0$ & $\mp x$ \\
$q(y,y,y)$ & $\pm y$ & $\pm y$ & $\pm y$ & $\pm y$ & $\pm y$ & $\pm y$ & $\pm y$ & $\pm y$ & $    0$ & $\mp y$ & $    0$ & $\mp y$ & $    0$ & $\pm x$ \\
\hline
\end{tabular}
\begin{tabular}{| l | c | c | c | c | c | c | c | c | c | c | c | c | c | c |  }
\hline
$q(x,x,x)$ & $\pm y$ & $\pm y$ & $\pm y$ & $\pm y$ & $\pm y$ & $\pm y$ & $\pm y$ & $\pm y$ & $\pm y$ & $\pm y$ & $\pm y$ & $\pm y$ & $\pm y$ & $\pm y$ \\
$q(x,x,y)$ & $\mp y$ & $\mp y$ & $\mp y$ & $\mp y$ & $\mp y$ & $\mp y$ & $\mp y$ & $\mp y$ & $\mp y$ & $\mp x$ & $\mp x$ & $    0$ & $    0$ & $    0$ \\
$q(x,y,x)$ & $    0$ & $    0$ & $    0$ & $\pm x$ & $\pm x$ & $\pm y$ & $\pm y$ & $\pm y$ & $\pm y$ & $\mp x$ & $\pm x$ & $\mp y$ & $\mp y$ & $\mp y$ \\
$q(x,y,y)$ & $\mp y$ & $    0$ & $\pm y$ & $\mp x$ & $\pm y$ & $\mp y$ & $    0$ & $\pm y$ & $\pm y$ & $\pm y$ & $\pm y$ & $\mp y$ & $    0$ & $\pm y$ \\
$q(y,x,x)$ & $    0$ & $    0$ & $    0$ & $\mp y$ & $\mp y$ & $    0$ & $    0$ & $\mp x$ & $    0$ & $\pm x$ & $\mp x$ & $    0$ & $    0$ & $    0$ \\
$q(y,x,y)$ & $    0$ & $    0$ & $    0$ & $\pm y$ & $\pm y$ & $    0$ & $    0$ & $\pm x$ & $    0$ & $\mp y$ & $\mp y$ & $    0$ & $    0$ & $    0$ \\
$q(y,y,x)$ & $    0$ & $    0$ & $    0$ & $\pm y$ & $\pm y$ & $    0$ & $    0$ & $\mp x$ & $    0$ & $\mp y$ & $\pm y$ & $    0$ & $    0$ & $    0$ \\
$q(y,y,y)$ & $    0$ & $    0$ & $    0$ & $\mp y$ & $\mp y$ & $    0$ & $    0$ & $\mp x$ & $    0$ & $\pm x$ & $\mp x$ & $    0$ & $    0$ & $    0$ \\
\hline
\end{tabular}
\begin{tabular}{| l | c | c | c | c | c | c | c | c | c | c | c | c | c | c |  }
\hline
$q(x,x,x)$ & $\pm y$ & $\pm y$ & $\pm y$ & $\pm y$ & $\pm y$ & $\pm y$ & $\pm y$ & $\pm y$ & $\pm y$ & $\pm y$ & $\pm y$ & $\pm y$ & $\pm y$ & $\pm y$ \\
$q(x,x,y)$ & $    0$ & $    0$ & $    0$ & $    0$ & $    0$ & $    0$ & $    0$ & $    0$ & $\pm x$ & $\pm x$ & $\pm x$ & $\pm x$ & $\pm x$ & $\pm x$ \\
$q(x,y,x)$ & $    0$ & $    0$ & $    0$ & $    0$ & $    0$ & $\pm y$ & $\pm y$ & $\pm y$ & $\mp y$ & $\mp y$ & $\mp x$ & $\pm x$ & $\pm x$ & $\pm x$ \\
$q(x,y,y)$ & $\mp y$ & $    0$ & $\pm y$ & $\pm y$ & $\pm y$ & $\mp y$ & $    0$ & $\pm y$ & $\mp x$ & $\pm y$ & $\pm y$ & $\mp x$ & $\pm x$ & $\pm y$ \\
$q(y,x,x)$ & $    0$ & $    0$ & $\mp x$ & $    0$ & $\pm x$ & $    0$ & $    0$ & $    0$ & $\mp y$ & $\mp y$ & $\mp x$ & $\mp y$ & $\pm y$ & $\mp y$ \\
$q(y,x,y)$ & $    0$ & $    0$ & $    0$ & $    0$ & $    0$ & $    0$ & $    0$ & $    0$ & $\pm y$ & $\pm y$ & $\pm y$ & $\pm y$ & $\pm y$ & $\pm y$ \\
$q(y,y,x)$ & $    0$ & $    0$ & $    0$ & $    0$ & $    0$ & $    0$ & $    0$ & $    0$ & $\pm y$ & $\pm y$ & $\mp y$ & $\pm y$ & $\pm y$ & $\pm y$ \\
$q(y,y,y)$ & $    0$ & $    0$ & $\mp x$ & $    0$ & $\pm x$ & $    0$ & $    0$ & $    0$ & $\mp y$ & $\mp y$ & $\mp x$ & $\mp y$ & $\pm y$ & $\mp y$ \\
\hline
\end{tabular}
\begin{tabular}{| l | c | c | c | c | c | c | c | c | c | c | c | c | c | c |  }
\hline
$q(x,x,x)$ & $\pm y$ & $\pm y$ & $\pm y$ & $\pm y$ & $\pm y$ & $\pm y$ & $\pm y$ & $\pm y$ & $\pm y$ & $\pm y$ & $\pm y$ & $\pm y$ & $\pm y$ & $\pm y$ \\
$q(x,x,y)$ & $\pm x$ & $\pm x$ & $\pm x$ & $\pm x$ & $\pm y$ & $\pm y$ & $\pm y$ & $\pm y$ & $\pm y$ & $\pm y$ & $\pm y$ & $\pm y$ & $\pm y$ & $\pm y$ \\
$q(x,y,x)$ & $\pm x$ & $\pm x$ & $\pm y$ & $\pm y$ & $\mp y$ & $\mp y$ & $\mp y$ & $\mp y$ & $    0$ & $    0$ & $    0$ & $\pm x$ & $\pm x$ & $\pm y$ \\
$q(x,y,y)$ & $\pm y$ & $\pm y$ & $\pm x$ & $\pm y$ & $\mp y$ & $    0$ & $\pm y$ & $\pm y$ & $\mp y$ & $    0$ & $\pm y$ & $\pm x$ & $\pm y$ & $\mp y$ \\
$q(y,x,x)$ & $\pm x$ & $\pm y$ & $\pm y$ & $\pm y$ & $    0$ & $    0$ & $\mp x$ & $    0$ & $    0$ & $    0$ & $    0$ & $\pm y$ & $\pm y$ & $    0$ \\
$q(y,x,y)$ & $\pm y$ & $\pm y$ & $\pm y$ & $\pm y$ & $    0$ & $    0$ & $\mp x$ & $    0$ & $    0$ & $    0$ & $    0$ & $\pm y$ & $\pm y$ & $    0$ \\
$q(y,y,x)$ & $\pm y$ & $\pm y$ & $\pm y$ & $\pm y$ & $    0$ & $    0$ & $\pm x$ & $    0$ & $    0$ & $    0$ & $    0$ & $\pm y$ & $\pm y$ & $    0$ \\
$q(y,y,y)$ & $\pm x$ & $\pm y$ & $\pm y$ & $\pm y$ & $    0$ & $    0$ & $\mp x$ & $    0$ & $    0$ & $    0$ & $    0$ & $\pm y$ & $\pm y$ & $    0$ \\
\hline
\end{tabular}
\begin{tabular}{| l | c | c | c | c | c |  }
\hline
$q(x,x,x)$ & $\pm y$ & $\pm y$ & $\pm y$ & $\pm y$ & $\pm y$\\
$q(x,x,y)$ & $\pm y$ & $\pm y$ & $\pm y$ & $\pm y$ & $\pm y$\\
$q(x,y,x)$ & $\pm y$ & $\pm y$ & $\pm y$ & $\pm y$ & $\pm y$\\
$q(x,y,y)$ & $    0$ & $\pm x$ & $\pm y$ & $\pm y$ & $\pm y$\\
$q(y,x,x)$ & $    0$ & $\pm y$ & $    0$ & $\pm x$ & $\pm y$\\
$q(y,x,y)$ & $    0$ & $\pm y$ & $    0$ & $\pm x$ & $\pm y$\\
$q(y,y,x)$ & $    0$ & $\pm y$ & $    0$ & $\pm x$ & $\pm y$\\
$q(y,y,y)$ & $    0$ & $\pm y$ & $    0$ & $\pm x$ & $\pm y$\\
\hline
\end{tabular}
\end{flushleft}
\normalsize

\end{proposition}


\noindent
\textsc{M.E and M. G; \\
Department of Mathematics, \\
University of South Florida, \\
4202 E Fowler Ave., \\
Tampa, FL 33620}\\
USA\\\\
and\\\\
\textsc{A. M;\\
Laboratoire de Math\'ematiques, \\
Informatique et Applications, \\
Universit\'e de Haute Alsace, \\
France.}

\end{document}